\documentclass[11pt,a4paper,reqno,tbtags]{article}
\usepackage{amsmath}
\usepackage{amsthm}
\usepackage{enumerate}
\usepackage{amssymb}
\usepackage{xpunctuate}

\usepackage{verbatim}
\usepackage{url}
\usepackage[latin1]{inputenc}

\usepackage{caption}
\usepackage{graphicx,color}
\usepackage{tikz}
\def\N{{\mathbb N}}
\def\Z{{\mathbb Z}}

\def\R{{\mathbb R}}

\def\Tb{{\mathbb T}}

\newcommand{\pr}[1]{\Pr\left(#1\right)}

\def\Fc{{\mathcal F}}

\def\hcX{{\widehat{\cX}}}

\def\org{\mathbf{0}}

\def\x{\mathbf{x}}
\def\y{\mathbf{y}}
\def\z{\mathbf{z}}
\def\0{\mathbf{0}}

\def\u{{\mathbf u}}
\def\v{{\mathbf v}}

\newcommand\normm[2]{\vertiii{#2}_{#1}}
\newcommand{\vertiii}[1]{{\left\vert\kern-0.25ex\left\vert\kern-0.25ex\left\vert #1  \right\vert\kern-0.25ex\right\vert\kern-0.25ex\right\vert}}

\def\be{\begin{equation}}
\def\ee{\end{equation}}
\def\bea{\begin{equation*}}
\def\eea{\end{equation*}}
\def\bal{\begin{aligned}}
\def\eal{\end{aligned}}
\def\ds{\displaystyle}
\def\eps{\varepsilon}

\newcommand\E{\operatorname{\mathbb E{}}}
\newcommand\PP{\operatorname{\mathbb P{}}}
\newcommand\Ex{\E}
\renewcommand\Pr{\PP}
\DeclareMathOperator{\Var}{Var}

\DeclareMathOperator{\ind}{{\bf 1}}

\def\d{\textup{d}}

\def\le{\leqslant}

\def\ge{\geqslant}

\newtheorem{thm}{Theorem}[section]
\newtheorem{lma}[thm]{Lemma}
\newtheorem{lemma}[thm]{Lemma}

\newtheorem{prop}[thm]{Proposition}

\newtheorem{claim}{Claim}

\theoremstyle{definition}
\newtheorem{remark}[thm]{Remark}

\theoremstyle{remark}

\newtheorem{preex}[thm]{Example}

\theoremstyle{definition}

\numberwithin{equation}{section}

\renewcommand\Re{\operatorname{Re}}

\begingroup
  \divide\count255 by 60
  \multiply\count255 by -60
  \advance\count255 by \time
  \ifnum \count255 < 10 \xdef\klockan{\the\count1.0\the\count255}
  \else\xdef\klockan{\the\count1.\the\count255}\fi
\endgroup

\newenvironment{romenumerate}[1][-10pt]{% optional argument changes indentation
\addtolength{\leftmargini}{#1}\begin{enumerate}% gives (i), (ii) etc.
 }{\end{enumerate}}

\newcounter{CC}
\newcommand{\CC}{\stepcounter{CC}\CCx} %new constant C_i
\newcommand{\CCx}{C_{\arabic{CC}}}     %repeats the last C_i
\newcommand{\CCdef}[1]{\xdef#1{\CCx}}     %defines #1 as the last C_i
\newcommand{\CCreset}{\setcounter{CC}0} %repeats from  C_1

\newcommand\gd{\delta}

\newcommand\gf{\varphi}
\newcommand\gam{\gamma}

\newcommand\gl{\lambda}

\newcommand\gs{\sigma}
\newcommand\gS{\Sigma}

\newcommand\gz{\zeta}

\newcommand\cB{\mathcal B}

\newcommand\cE{\mathcal E}
\newcommand\cF{\mathcal F}

\newcommand\cP{\mathcal P}

\newcommand\cR{{\mathcal R}}

\newcommand\cX{{\mathcal X}}
\newcommand\cY{{\mathcal Y}}
\newcommand\cZ{{\mathcal Z}}

\newcommand\bbN{\mathbb N}

\newcommand\bbZ{\mathbb Z}

\newcommand\set[1]{\ensuremath{\{#1\}}}
\newcommand\bigset[1]{\ensuremath{\bigl\{#1\bigr\}}}

\newcommand\xpar[1]{(#1)}
\newcommand\bigpar[1]{\bigl(#1\bigr)}
\newcommand\Bigpar[1]{\Bigl(#1\Bigr)}

\newcommand\sqpar[1]{[#1]}
\newcommand\bigsqpar[1]{\bigl[#1\bigr]}
\newcommand\Bigsqpar[1]{\Bigl[#1\Bigr]}

\newcommand\lrsqpar[1]{\left[#1\right]}

\newcommand\bigabs[1]{\bigl\lvert#1\bigr\rvert}
\newcommand\Bigabs[1]{\Bigl\lvert#1\Bigr\rvert}

\newcommand\lrabs[1]{\left\lvert#1\right\rvert}
\newcommand\norm[1]{\lVert#1\rVert}

\newcommand\bx{\mathbf{x}}
\newcommand\bo{\mathbf{0}}

\newcommand\bz{\mathbf{z}}
\newcommand\xcpar[1]{\{#1\}}
\newcommand\indic[1]{\boldsymbol1\xcpar{#1}}

\newcommand{\refT}[1]{Theorem~\ref{#1}}

\newcommand{\refL}[1]{Lemma~\ref{#1}}
\newcommand{\refLs}[1]{Lemmas~\ref{#1}}
\newcommand{\refR}[1]{Remark~\ref{#1}}

\newcommand{\refP}[1]{Proposition~\ref{#1}}
\newcommand{\refS}[1]{Section~\ref{#1}}

\newcommand{\refSS}[1]{Section~\ref{#1}}

\newcommand\qw{^{-1}}

\newcommand\qq{^{1/2}}

\newcommand\punkt{\xperiod}    % xpunctuate
\newcommand\iid{i.i.d\punkt}    
\newcommand\ie{i.e\punkt}
\newcommand\eg{e.g\punkt}

\newcommand\cf{cf\punkt}
\newcommand{\as}{a.s\punkt}

\newcommand{\tend}{\longrightarrow}

\newcommand\asto{\overset{\mathrm{a.s.}}{\tend}}
\newcommand\eqd{\overset{\mathrm{d}}{=}}

\newcommand\QLE[2]{#1^{\le #2}}
\newcommand\QGT[2]{#1^{> #2}}
\newcommand\QGLE[3]{#1^{(#2,#3]}}
\newcommand\QLEz[1]{\QLE{\bgz}{#1}}
\newcommand\QGTz[1]{\QGT{\bgz}{#1}}

\newcommand\bgzi{\QGLE{\bgz}{r_i}{r_{i+1}}}
 \newcommand\bgzp{\bgz^+}
\newcommand\Yp{Y^p}
\newcommand\Ypq[1]{Y^p_{#1}}
\newcommand\Ypo{\Ypq{\0}}
\newcommand\Ypx{\Ypq{\x}}

\newcommand\Yx[1]{Y^{#1}}
\newcommand\Yhalf{\Yx{1/2}}
\newcommand\cYp{\cY^p}

\newcommand\etaz{\eta_{\z}}
\newcommand\cXq[1]{\cX_{#1}}
\newcommand\cXo{\cXq{\0}}

\newcommand\cXz{\cXq{\z}}
\newcommand\Xqq[2]{X_{#1,#2}}
\newcommand\Xzq[1]{\Xqq{\z}{#1}}

\newcommand\Xzx{\Xzq{\x}}

\newcommand\Zd{\Z^d}
\newcommand\Td{\Tb^d}
\newcommand\bgz{{\boldsymbol{\gz}}}
 \newcommand\bfeta{\boldsymbol\eta}
\newcommand\Mx{M^\ast}
\newcommand\hmu{\widehat\mu}
\newcommand\sumzzd{\sum_{\z\in\Zd}}

\newcommand\pfitemx[1]{\par#1:}
\newcommand\pfitemref[1]{\pfitemx{\ref{#1}}}
\newcommand\EVar[1]{\Var\lrsqpar{#1}}
\newcommand\EEVar[1]{\E\lrsqpar{#1^2}}
\newcommand\ntoo{\ensuremath{{n\to\infty}}}
\newcommand\rhou{\overline{\rho}}
\newcommand\rhol{\underline{\rho}}

\newcommand\Zo{Z_\0} % Z_\0 in sec:nonfix
\newcommand\gsf{$\gs$-field}

\newcommand\cZx{\cZ^\ast}
\newcommand\ActR{\mathcal{A}^R}
\newcommand\ActB{\mathcal{A}^B}
\newcommand\ppr{^{+,r}}
\newcommand\mmr{^{-,r}}
\newcommand\Br{B(\0,r)}
\newcommand\qmc{quasi-monochromatic}
\begin{document}

\title{To fixate or not to fixate in two-type annihilating\\ branching
  random walks}
\date{15 October, 2020}
\author{Daniel Ahlberg, Simon Griffiths and Svante Janson
%but not Robert Morris
}

\maketitle

\begin{abstract}
We study a model of competition between two types evolving as branching random walks on $\Z^d$. The two types are represented by red and blue balls respectively, with the rule that balls of different colour annihilate upon contact. We consider initial configurations in which the sites of $\Z^d$ contain one ball each, which are independently coloured red with probability $p$ and blue otherwise. We address the question of \emph{fixation}, referring to the sites eventually settling for a given colour, or not.
Under a mild moment condition on the branching rule, we prove that the process will fixate almost surely for $p\neq 1/2$, and that every site will change colour infinitely often almost surely for the balanced initial condition $p=1/2$.
\end{abstract}

\section{Introduction}\label{Sintro}

Position, on each site of a connected graph $G$, an urn. Each urn may contain either red or blue balls, but not both at once. At the dawn of time ($t=0$), red and blue balls are distributed in the urns according to some rule. The balls come equipped with unit-rate Poisson clocks, and when a clock rings, the corresponding ball immediately sends an independent copy of itself to each of the urns at neighbouring sites (while the ball with the clock remains where it is). As red and blue balls may not exist together in the same urn, they annihilate on a one-to-one basis.

In the case that $G$ is connected and \emph{finite}, the authors together with Morris~\cite{ahlgrijanmor19} have proved that the system of urns will eventually almost surely contain balls of only one colour. In the current paper we examine the process on the $d$-dimensional integer lattice $\Z^d$, for $d\ge1$, evolving from an initial configuration with a ball at each site, which independently from one another are coloured red with probability $p$ and blue otherwise. Our main result shows that for $p\neq1/2$ the colouring of the lattice induced by the urn process eventually \emph{fixates} almost surely on a single colour, and that for $p=1/2$ each site almost surely switches colour infinitely many times. We shall prove our results for a general family of branching mechanisms, further described below, of which the above mentioned nearest-neighbour rule is merely one example.

%As a matter of fact, our main result holds in a more general setting in which the nearest-neighbour rule is replaced by a more general branching mechanism. In the monochromatic setting, in which all initial balls have the same colour, and thus all later balls too, so there are no annihilations, the process corresponds to a continuous time branching random walk on the integer lattice. So the competition model we consider could be referred to as (two-type) annihilating branching random walk.

Models for systems of particles annihilating upon contact have a long history. The question of site recurrence in a one-dimensional system of (non-branching) random walkers annihilating upon contact was raised in the mid 1970s by Erd\H os and Ney~\cite{erdney74}.
%Models of random walks annihilating upon contact began to appear at the same period of time. The question of site recurrence, whether every site is visited by a particle at arbitrarily large times, in a one-type system in one dimension was raised by Erd\H os and Ney~\cite{erdney74}.
Higher dimensional versions of the same problem was soon after considered by Griffeath~\cite{griffeath78} and Arratia~\cite{arratia83}. These problems concern a system of particles of a single type.
Analogous models consisting of two types of particles have been suggested in the physics literature as descriptive for the inert chemical reaction $A+B\to\emptyset$, see e.g.~\cite{ovczel78,touwil83}. These models tend to require a different set of techniques for their analysis. In this setting, Bramson and Lebowitz~\cite{braleb91a,braleb91b} derived the rate of decay of the density of particles in such a two-type model where particles perform simple random walks, and particles of different type annihilate upon contact.
In more recent work, Cabezas, Rolla and Sidoravicius~\cite{cabrolsid18} addressed site recurrence in a similar setting, and proved that the origin is visited at arbitrarily large times.
A related model has been considered by Damron, Gravner, Junge, Lyu and Sivakoff~\cite{damgrajunlyusiv19}.
%Related models have been studied with exclusion dynamics on tori~\cite{sasada10}.

The (discrete time) branching random walk first arose as a geometric interpretation of the evolution of generations in an age-dependent branching process, in work of Kingman~\cite{kingman75}, Biggins~\cite{biggins76,biggins77} and Bramson~\cite{bramson78}. Later work has explored important connections between branching random walks and their continuum counterpart, branching Brownian motion, to central objects in statistical physics such as spin glasses and the discrete planar Gaussian free field. For a more detailed discussion on these models and connections, we refer the reader to the monographs~\cite{bovier15,shi15,zeitouni16}.

Survival for a version of the branching random walk, where any two particles
annihilate upon contact, was studied by Bramson and
Gray~\cite{bragra85}.
Very much in spirit of their and other authors' work (cited above), and
further motivated by the corresponding question for Glauber dynamics of the
Ising model (see, e.g.,~\cite{fonschsid02,morris11}), we here address
fixation for the two-type annihilating urn system on $\Z^d$ starting from a
stationary random initial configuration.
%It is questions of this type that we shall address for our two-type annihilating branching random walk.
In the monochromatic setting, in which all balls have the same colour, so
there are no annihilations, the process we study corresponds to a continuous
time branching random walk on the integer lattice. For this reason, we shall
interchangeably refer to the model we consider as a \emph{competing urn scheme}
and as a \emph{two-type annihilating branching random walk}.  
Our analysis of this process will in large parts be based on a combination of martingale techniques, and elements of Fourier analysis.

\subsection{Model and results}

We proceed with a somewhat more formal description of the model we consider,
and introduce some notation. We shall encode the presence of a red ball with
the value $+1$ and the presence of a blue ball by the value $-1$. This
encoding produces a bijection between particle configurations and
integer-valued vectors indexed by $\Z^d$. Below, a \emph{configuration} on
$\Z^d$ will refer to a vector $\bgz=(\zeta_\z)_{\z\in\Z^d}$ of integers. A
configuration $\bgz$ is said to be \emph{locally finite} if
$|\zeta_\z|<\infty$ for all $\z\in\Z^d$, and \emph{finite} if the total
number of particles $\|\bgz\|:=\sum_{\z\in\Z^d}|\zeta_\z|$ is finite. (For
typographical convenience, we occasionally write $\zeta(\z)$ for
$\zeta_\z$.)

Let $\Phi$ be a probability measure on finite non-negative configurations on
$\Z^d$, and let $\varphi$ denote a generic random configuration distributed
according to $\Phi$. Given a locally finite initial configuration $\bgz$, we
assign to each ball in $\bgz$ a clock independent from everything else. At
the ring of a clock at position $\z$, the corresponding ball makes an independent draw from
the distribution $\Phi$, and positions new balls accordingly, translated by
$\z$. (The ball with the clock is assumed to remain where it is, although we
shall comment on this restriction below, in Remark~\ref{r:death}.)
As before, all children have the same colour as their parent,
 and if one or several balls are positioned in an urn with balls of opposite
colour, then they immediately 
annihilate one for one until all remaining balls in the urn are of
the same colour. The nearest-neighbour rule (from~\cite{ahlgrijanmor19})
described above thus corresponds to $\Phi$ being the degenerate measure
supported on the configuration consisting of one ball at each of the $2d$
neighbours to the origin.

Since brevity is the soul of wit, we have here chosen to be brief;
we shall at later occasions in the text have more to say about the
construction of the process as the need arises.

We shall henceforth impose two restrictions on $\Phi$.
We say that $\Phi$ is \emph{irreducible}\footnote{%
Assuming that the
  offspring distribution is irreducible means no loss of generality, since
  we otherwise could consider the process on the subgroup $G$ of $\Z^d$
  generated by the support of $\Phi$; note that
$G\cong \Z^{d_1}$ for some $d_1\le d$ and that
the process on $\Z^d$ then decomposes into independent copies of the process
on $G$, supported on different translates (cosets) of $G$.
We ignore the trivial case when the support of $\Phi$ is $\set{\0}$;
then the urns are independent continuous-time branching processes, each with
a fixed colour.
} if it is not
supported on a proper subgroup of $\Z^d$.
We will assume throughout that $\Phi$ is irreducible
and that $\|\varphi\|$ has finite mean, so that
\be\label{eq:phi:mean}
0 < \lambda := \E[ \| \varphi \| ] < \infty.
\ee

\begin{remark}\label{Rfinite}
  We will for simplicity
only consider initial configurations with at most one ball at each
  site. 
(Although more general cases might also be interesting, see \refS{sec:open}.)
  In this case, and assuming \eqref{eq:phi:mean},  the
  process described above is well-defined and without explosions;
  more precisely, for any finite box $B(\0,r):=[-r,r]^d$ and any finite $T$,
  there is \as{} (almost surely) only a finite number of balls
  appearing in $B(\0,r)$ at some time in $[0,T]$, and as a consequence
  there is only a
  finite number of nucleations (branching events) and annihilations at any given site in a
  finite time interval.
  It is straightforward to verify these claims for any finite initial configuration.
  For monochromatic (possibly infinite) initial configurations
  these claims follow since
  the expected number of balls at any given site at time $t$ is at most
  $e^{\gl t}<\infty$, \cf{} \eqref{jb} with $p=1$.
In general,  
  as we detail in Section~\ref{sec:existence},
  we may formally 
define the annihilating process for arbitrary initial configurations
  as the \as{}\ limit of processes with
  finite initial configurations; the claimed properties are shown to carry over from the finite setting.
In addition to the above, since the process has only a finite number of jumps at each site
in each finite time interval, we may assume the standard convention
that the process is right-continuous with left limits.
\end{remark}

We aim in this paper to understand the evolution of the
annihilating system on $\Z^d$ starting from a stationary random initial
configuration. To be precise, given $p \in [0,1]$ and $d \ge 1$, define a
\emph{$p$-random Bernoulli colouring} of $\Z^d$ as follows: for each $\z \in
\Z^d$, the corresponding urn initially contains a single red ball with
probability $p$, and a single blue ball otherwise, all independently.
Hence, the resulting configuration corresponds to an element in $\{-1,1\}^{\Z^d}$.
We say that the two-type annihilating branching random walk 
\emph{fixates}\footnote{%
The times the different urns fixate are random and different; 
we do not claim that there is a single time when all urn have the same
colour.
(Indeed, since the system is infinite, we cannot expect this.)
} if there exists a colour $c$ such
that every urn eventually contains only balls of colour $c$.

As a measure on the displacement of balls in each nucleation, we define
for $r>0$,
\begin{align}\label{fir}
  \normm{r}{\varphi}\,:=\,\sum_{\z\in\Z^d}|\z|^r\varphi(\z).
\end{align}

Our main theorem is the following.
(By symmetry, it suffices to consider $p\ge\frac12$.)

\begin{thm}\label{thm:fixnonfix}
Let $d \in \N$, and let $\Phi$ be an irreducible probability measure on
finite configurations on $\Z^d$ such that 
$\E\big[\normm{1}{\varphi}^2\big]<\infty$,
$\E\big[\normm{2}{\varphi}\big]<\infty$, and
either $\E[\|\varphi\|^3]<\infty$ if $d=1$ or $\E[\|\varphi\|^2]<\infty$ if $d\ge2$.
Then, for the competing urn scheme on $\Z^d$ starting from a $p$-random
Bernoulli colouring, almost surely:
\begin{romenumerate}
\item\label{thm:fixnonfixa}
For $p>\frac12$ the system fixates;
each urn is eventually red. 
%Furthermore, the density of red urns, as defined in \eqref{eq:density},
%exists for all $t\ge0$ and tends to $1$ as $t\to\infty$.
  \item\label{thm:fixnonfixb}
    For $p=\frac12$ every site changes colour infinitely often.
\end{romenumerate}
\end{thm}

Furthermore, we 
define the density of red sites at time $t$ as the limit (if it exists)
\begin{equation}\label{eq:density}
\rho(t):=
\lim_{n\to\infty}\frac{1}{(2n+1)^d}\sum_{\z\in[-n,n]^d}\indic{\z\text{ is red at time }t},
\end{equation}
and show, in Section~\ref{sec:density}, that almost surely the limit exists
for all $t\ge0$ and satisfies $\rho(t)=\Pr(Z_\org(t)>0)$. In particular, for
$p>\frac12$, it follows that the density of red urns \as{} tends to $1$ as
$t\to\infty$. It is remarkable that for $p=\frac12$ our arguments do not
show that $\rho(t)\to\frac12$ as $t\to\infty$; see Section~\ref{sec:open}
for a more precise conjecture.

\begin{remark}\label{R4+eps}
For part \ref{thm:fixnonfixb} of Theorem~\ref{thm:fixnonfix} the condition
$\E[\|\varphi\|^2]<\infty$ suffices for all (irreducible) offspring distributions in
all dimensions $d\ge1$.
We do not know if the same condition suffices also for $d=1$ in
part~\ref{thm:fixnonfixa} of the theorem, or for Theorem~\ref{thm:Zd-limit} below,
on which the proof of part~\ref{thm:fixnonfixa} is based.
For $L^2$ convergence in Theorem~\ref{thm:Zd-limit}, the finite second moment condition
suffices for all (irreducible) offspring distributions in all dimensions
$d\ge1$.
We do not know whether the conditions
\be\label{eq:displacement}
\E\big[\normm{1}{\varphi}^2\big]<\infty
\quad\text{and}\quad
\E\big[\normm{2}{\varphi}\big]<\infty
\ee
on the spatial displacement that figure in our theorems are necessary.
\end{remark}

\begin{remark}\label{r:death}
In the definition of the model, as offspring is produced, the parent is assumed to remain where it is.
As a result, in the monochromatic version of the process, once a ball is born it remains in the 
same place at all future times. More generally we could assume that each ball lives for an exponentially
 distributed life time, at the end of which it reproduces according to $\Phi$ and disappears. This is 
 certainly more general, as $\Phi$ could be specified to produce a copy of the parent in its place with 
 probability one. In addition,
 %By choosing $\Phi$ accordingly,
 this allows us, for instance, to consider models where the balls move according to continuous time random walks, which in each step branch with a non-zero probability.
We shall in Section~\ref{s:death} describe how our results can be extended to cover also this setting.
\end{remark}

Central in order to understand the annihilating process will be to closely examine the evolution of
the monochromatic process (without annihilations), in which each site of $\Z^d$ independently is
initially occupied by a particle with probability $p \in (0,1]$ and
otherwise empty.
Let $\cYp(t)=(\Yp_{\z}(t))_{\z\in\Z^d}$ 
be the configuration at time $t\ge0$ of this process, where thus $(\Yp_{\z}(0))_{\z\in\Z^d}$ are \iid{} Bernoulli with parameter $p$.

We shall prove the following result on the asymptotics of the monochromatic
system.

\begin{thm}\label{thm:Zd-limit}
Let $d \in \N$, and let $\Phi$ be an irreducible probability measure on
finite configurations on $\Z^d$ such that 
$\E\big[\normm{1}{\varphi}^2\big]<\infty$,
$\E\big[\normm{2}{\varphi}\big]<\infty$, and
either $\E[\|\varphi\|^3]<\infty$ if $d=1$ or $\E[\|\varphi\|^2]<\infty$ if $d\ge2$.
Then, for every $p \in [0,1]$ and every $\z\in\Zd$,
we have, with $\gl$ given by \eqref{eq:phi:mean},
\begin{equation}\label{tYlimit}
\lim_{t \to \infty} e^{-\lambda t} \Yp_\z(t) \,=\, p  
\end{equation}
almost surely and in $L^2$.
\end{thm}

\subsection{Outlines of proof and paper}

In the monochromatic process balls do not interact with each other, and in order to understand its asymptotics it will suffice to examine the evolution of each ball initially present in the system separately. For this purpose we let $\cXz(t)=(\Xzx(t))_{\x\in\Z^d}$ be the configuration at
time $t$ of the process started with a single ball at $\z$, \ie,
$\Xzx(0)=\gd_{\x,\z}$.
(These processes are obviously just translates of $\cXo(t)$, the evolution of a single ball started at the origin, but the
collection of all of them will be useful in our arguments.)
Note that the process $(\cXz(t))_{t\ge0}$ is a multi-type continuous time
Markov branching process with type space $\Z^d$; see e.g.\
\cite[Section~V.7]{athney72}. % or \cite[Section~6.1]{jagers75}. 
Moreover, the
dynamics of the process is translation invariant, which in particular
implies that $\norm{\cXz(t)}$, the total number of balls in the system, 
evolves as a
(single-type) continuous time Markov branching process 
in which 
each individual gets $\norm{\gf}$ children
with rate 1.
The finite moment
condition~\eqref{eq:phi:mean} is well-known to imply that the process is
almost surely finite at all times, see \cite[Section III.2]{athney72};
in fact, 
it is easily seen that
$e^{-\gl t}\norm{\cXz(t)}$ is a martingale, and thus, in particular,
\begin{align}\label{ex}
\E\norm{\cXz(t)}=e^{\gl t} \E\norm{\cXz(0)} = e^{\gl t},
\end{align}
see e.g.~\cite[Section III.4 and Theorem III.7.1]{athney72} or
Lemma~\ref{lma:martingale} below.

A large part of our work will consist in exploring the evolution of the process $\cXo(t)$ starting with a single
ball in \refS{sec:oneball}, and its implications for the monochromatic process $\cYp(t)$, which is studied in \refS{sec:monochromatic},
leading to a proof of \refT{thm:Zd-limit}. 
The analysis will be based on martingale techniques and elements of Fourier
analysis. 
An important step in the argument is a precise variance estimate, which is stated as \refP{propY}\ref{prop:var}.

We then return to the two-colour competition process, starting from a
$p$-random Bernoulli colouring, which we describe by the vector
$\cZ(t)=(Z_\bx(t))_{\bx\in\bbZ^d}$.
Although our main interest lies in the case of a random initial configuration,
we will in the proofs consider various versions, and we thus 
allow an arbitrary initial configuration
$\bgz=(\gz_\bx)_{\bx\in\bbZ^d}\in \set{-1,0,1}^{\bbZ^d}$, deterministic or random.
(Thus $\gz_\bx=-1$ means a blue ball at $\bx$, 1 means a red ball and 0 means
no ball.)
We let $\cZ(t,\bgz)$, for $t\ge0$, denote the process started from $\bgz$.
In particular, the monochromatic process $\cYp(t)$ equals (as a process)
$\cZ(t,\bgz)$ with $\bgz=(\zeta_\x)_{\x\in\Zd}$ independent Bernoulli with parameter $p$,
and $\cXz(t)$ corresponds to $\cZ(t,\bgz)$ with $\bgz=(\gd_{\bx,\bz})_{\bx\in\Zd}$, whereas $\cZ(t)$ itself corresponds to $\cZ(t,\bgz)$ with $\bgz$ being the $p$-random Bernoulli colouring whose entries are $\pm1$-valued and independent from one another.

We describe in \refS{sec:existence} how the annihilating process can be defined in a formal fashion, 
and derive along the way some properties that will be used in the proof of our main theorem.
One such preliminary result is a coupling, previously
employed in \cite{ahlgrijanmor19}, 
that enables us to ignore annihilations and instead study a
pair of (dependent) monochromatic processes, to which we can apply results
from previous sections.
Theorem~\ref{thm:fixnonfix}\ref{thm:fixnonfixa} 
then is as an easy
consequence of Theorem~\ref{thm:Zd-limit}. 
The balanced case, Theorem~\ref{thm:fixnonfix}\ref{thm:fixnonfixb}, will require a
finer analysis of the order of fluctuations of 
the monochromatic process, which suitably comes out as a side while proving
Theorem~\ref{thm:Zd-limit}, together with a decoupling argument showing that the
states of any finite set of sites are irrelevant for the long term evolution.
Details are given in Section~\ref{sec:fixnonfix}.

At the end of this paper, in \refS{sec:density}, we show that the density as defined in~\eqref{eq:density}
is well-defined, and in \refS{s:death}, we describe how to adapt our
arguments to cover the more general version of our process where balls are
assumed to die as they reproduce, cf.\ \refR{r:death}. 
Finally, \refS{sec:open} contains some further directions and open problems.

\section{The evolution of a single ball}\label{sec:oneball}

In this section we analyze the evolution of a single ball,
i.e., the process $\cXz(t)$. By translation invariance, we may without loss of generality assume $\z=\0$.
Furthermore, in this section (only), we drop the index $\z$ indicating the
starting position and
use
the notation $\cX(t)=(X_\x(t))_{\x\in\Zd}$ for $\cXo(t)$.
The analysis is based on a combination of
Fourier analysis and a martingale approach.
We assume throughout that~\eqref{eq:phi:mean} holds.

\subsection{Elements of Fourier analysis}

We proceed with the study of the process $(\cX(t))_{t\ge0}$, evolving from
a single ball initially at the origin. Recall that
$\cX(t)=(X_{\x}(t))_{\x\in\Z^d}$, for each $t\ge0$, is an almost surely
finite configuration on $\Z^d$. From a harmonic analysis point of view, the
dual group of $\Z^d$ is the cycle group $\Tb^d$, which we identify with
$(-\pi,\pi]^d$. Hence, there is a natural correspondence between
configurations on $\Z^d$ and certain complex-valued functions on $\Tb^d$. More
precisely, we define the \emph{Fourier transform} of $\cX(t)$ as 
\be\label{eq:Fourier}
\hcX_\u(t):=\sum_{\x\in\Z^d}e^{i\u\cdot\x}X_{\x}(t),\quad\u\in\Tb^d.
\ee
We note that for $t=0$ this definition yields $\hcX_\u(0)=1$, and for $\u=\0$ we obtain
\be\label{eq:hcX0}
\hcX_\0(t)=\sum_{\x\in\Z^d}X_{\x}(t)=\|\cX(t)\|,
\ee
the total number of balls at time $t$. In general, the inequality
$|\hcX_\u(t)|\le\|\cX(t)\|$ remains valid. 
$\cX(t)$ may be recovered via the \emph{inversion formula}:
\be\label{eq:inversion}
X_{\x}(t)=\int_{\Tb^d}e^{-i\u\cdot\x}\hcX_\u(t)\, \d\u,
\ee
where $\d\u$ denotes the normalized Lebesgue measure
$(2\pi)^{-d}\,\d\u_1\cdots \d\u_d$ on $\Tb^d$.
(This  is easily checked; plug in~\eqref{eq:Fourier} and compute the integral.)

Denote by $\mu=(\mu(\x))_{\x\in\Z^d}$ the coordinate-wise expectation of
$\Phi$, i.e.\ $\mu(\x):=\E[\varphi(\x)]$. 
Then, by \eqref{eq:phi:mean},
\begin{align}\label{eq:phi:mean-mu}
\norm{\mu}  
:=\sum_{\x\in\Zd} \mu(\x)
=\sum_{\x\in\Zd}\E \gf(\x)
=\E\sum_{\x\in\Zd} \gf(\x)
=\E\norm{\gf}=\gl <\infty.
\end{align}
Hence, also $\mu$ 
has a well-defined Fourier transform
$\widehat\mu(\u):=\sum_{\x\in\Zd}e^{i\u\cdot\x}\mu(\x)$;
note that 
\begin{align}\label{mu0}
\widehat\mu(\0)=\|\mu\|=\lambda.   
\end{align}
(The Fourier transform $\widehat{\zeta}(\u)$ of any finite configuration $\zeta$ on $\Z^d$ is defined analogously.)

As said in the introduction, it is well-known that 
$e^{-\gl t}\hcX_\0(t)=e^{-\gl t}\|\cX(t)\|$ is a  continuous time
martingale.
We extend this to arbitrary $\u\in\Tb^d$ in the next lemma. Let 
\begin{align}
  \label{Mu}
M_\u(t):=e^{-\widehat\mu(\u)t}\hcX_\u(t).
\end{align}
In particular, by \eqref{mu0} and \eqref{eq:hcX0},
\begin{align}
  \label{M0}
M_\0(t):=e^{-\gl t}\norm{\cX(t)}.
\end{align}

\begin{lma}\label{lma:martingale}
The process $(M_\u(t))_{t\ge0}$ is a martingale for each $\u\in\Tb^d$. In
particular, 
\be\label{eq:Fmalthusian}
\E[\hcX_\u(t)]=e^{\widehat\mu(\u)t},\quad\u\in\Tb^d.
\ee
\end{lma}

Taking $\u=\0$ in \refL{lma:martingale} %\eqref{eq:Fmalthusian}, 
we recover, using \eqref{M0}, the fact noted above that
$e^{-\gl t}\norm{\cX(t)}$ is a martingale, and in particular, 
since also $\norm{\cX(0)}=1$,
that
\be\label{eq:malthusian}
\E[\|\cX(t)\|]=e^{\lambda t},
\ee
i.e., that \eqref{ex} holds.

\begin{proof}[Proof of Lemma~\ref{lma:martingale}]
We prove first~\eqref{eq:Fmalthusian}. Once~\eqref{eq:Fmalthusian} has been
proven the martingale property will follow from
the Markov and branching properties together with 
homogeneity in time and space.
Hence, it will suffice to prove~\eqref{eq:Fmalthusian}. 

Note that, almost surely, no two clocks ever ring at the same time. If the
clock rings for a ball at $\z$, then $\cX(t)$ jumps by (a copy of) $\varphi$
translated by the vector $\z$. Hence, $\hcX_\u(t)$ then jumps by
\be\label{eq:hcX_jump}
\Delta\hcX_\u(t)\,=\,\sum_{\y\in\Z^d}e^{i \u\cdot(\z+\y)}\varphi(\y)\,=\,e^{i\u\cdot\z}\widehat\varphi(\u),
\ee
and the expected jump of $\hcX_\u(t)$, given that the clock rings for a ball
at $\z$, is
\begin{align}
  e^{i\u\cdot\z}\E[\widehat\varphi(\u)]\,=\,e^{i\u\cdot\z}\widehat\mu(\u),
\end{align}
which by~\eqref{eq:phi:mean-mu} is finite.
Since the number of balls at $\z$ is $X_{\z}(t)$, and each rings with
intensity 1, this implies
\begin{align}
  \frac{\d}{\d t} \E[\hcX_\u(t)]
\,=\,\sum_{\z\in\Z^d}\E[X_{\z}(t)]e^{i\u\cdot\z}\widehat\mu(\u)
\,=\,\widehat\mu(\u)\E[\hcX_\u(t)],
\end{align}
and~\eqref{eq:Fmalthusian} follows by the initial condition $\hcX_\u(0)=1$.
\end{proof}

As another consequence of \refL{lma:martingale},
we obtain a formula for the expected number of balls at a given
position. Since $|\hcX_\u(t)|\le\|\cX(t)\|$ and 
$\E\norm{\hcX(t)}<\infty$,
we may combine the inversion
formula~\eqref{eq:inversion}, Fubini's theorem and \eqref{eq:Fmalthusian} 
to obtain the expression 
\be\label{eq:exp_inversion}
\E[X_{\z}(t)]
=\int_{\Tb^d}e^{-i\u\cdot\z}\E [\hcX_{\u}(t)]\,\d\u
=\int_{\Tb^d}e^{-i\u\cdot\z}e^{\widehat\mu(\u)t}\,\d\u.
\ee

\subsection{Second moment analysis}

To obtain higher moments of the process we will require a stronger
assumption on the moments of $\Phi$. This is also where
condition~\eqref{eq:displacement} on the displacement of $\Phi$ comes
in. 
%More precisely shall we prove the following.

We begin by noting that
the condition $\E[\|\varphi\|^2]<\infty$ implies that
$\E[\|\cX(t)\|^2]<\infty$ for all $t\ge0$, see
\cite[Corollary~III.6.1]{athney72} or \cite[Theorem~6.3.6]{jagers75}. Since
$|\hcX_\u(t)|\le\|\cX(t)\|$ we have as a consequence that
$\E[M_\u(t)^2]<\infty$ for all $\u\in\Tb^d$ and $t\ge0$;
in other words, $M_\u(t)$ is a square-integrable martingale for every
$\u\in\Td$.
The following proposition shows that 
under the
condition $\Re\widehat\mu(\u)>\frac12\lambda$, this martingale is $L^2$-bounded. % in $t$.

\begin{prop}\label{prop:higher}
Assume that $\E[\|\varphi\|^2]<\infty$,
and let  $\u\in\Tb^d$ be such that
$\Re\widehat\mu(\u)>\frac12\lambda$.
Then the process
$(M_\u(t))_{t\ge0}$ is an $L^2$-bounded martingale; in particular, the limit
$\Mx_\u:=\lim_{t\to\infty}M_\u(t)$ exists almost surely and in $L^2$.
Furthermore, there exists a constant $C(\u)$, which is 
uniformly bounded for $\Re\widehat\mu(\u)-\tfrac12\lambda\ge c$ for any $c>0$,
such that for all $t\ge0$
%with $C(\u):=(2\Re\hat\mu(\u)-\lambda)^{-1}$,
\be\label{eq:higher1}
\E\big[|M_\u(t)-\Mx_\u|^2\big]\,\le\,
C(\u)\E[\|\varphi\|^2]\,e^{-(2\Re\widehat\mu(\u)-\lambda)t},
\ee
and if, in addition, 
$\E\big[\normm{1}{\varphi}^2\big]<\infty$ and
$\E\big[\normm{2}{\varphi}\big]<\infty$, then for all $t\ge0$ 
\be\label{eq:higher2}
\E\big[|M_\u(t)-M_\0(t)|^2\big]\,\le\,C(\u)|\u|^2.
\ee
%where the implicit constant in the right-hand side is uniform for 
%$\Re\hat\mu(\u)-\tfrac12\lambda\ge c$ for any $c>0$.
\end{prop}

The following lemma will be the first step towards the above proposition.

\begin{lma}\label{lma:higher}
Suppose $\E[\|\varphi\|^2]<\infty$.
For every $\u,\v\in\Tb^d$ and $t\ge0$ we have
\begin{align}
  \E\big[M_\u(t)M_\v(t)\big]
=1+\E\big[\widehat\varphi(\u)\widehat\varphi(\v)\big]\int_0^te^{\widehat\mu(\u+\v)x-[\widehat\mu(\u)+\widehat\mu(\v)]x}\,\d x.
\end{align}
\end{lma}

\begin{proof}
The two processes $(M_\u(t))_{t\ge0}$ and $(M_\v(t))_{t\ge0}$ are square-integrable martingales. 
Hence, their quadratic covariation 
$[M_\u,M_\v](t)$ is well-defined; see e.g.~\cite[Section II.6]{protter90}. 
(Although not needed here, we note that it may be defined as the following limit,
in probability \cite[Theorem II.23]{protter90},
\begin{align}
  [M_\u,M_\v](t):=
M_\u(0)M_\v(0)+
\lim_{|P_n|\to0}\sum_{k=1}^n\bigl(M_\u(t_k)-M_\u(t_{k-1})\bigr)
\bigl(M_\v(t_k)-M_\v(t_{k-1})\bigr),
\end{align}
where $(P_n)_{n\ge1}$ is some sequence of partitions 
$0=t_{0}<t_{1}<\dots<t_{n}=t$
of $[0,t]$ with mesh 
 $\max_k|t_k-t_{k-1}|$
tending to zero.)

The process $\{M_\u(t)M_\v(t)-[M_\u,M_\v](t):t\ge0\}$ is again a martingale
\cite[Corollary 2 to Theorem II.27]{protter90}, 
which vanishes at $t=0$ by definition,
and thus
\be\label{eq:XY-[XY]}
\E\big[M_\u(t)M_\v(t)\big]%-\E\big[M_\u(0)M_\v(0)\big]
\,=\,\E\big[[M_\u,M_\v](t)\big]\quad\text{for all }t\ge0.
\ee
Furthermore, $M_\u(t)$ and $M_\v(t)$
have finite variation on each compact time interval (since each realisation has
piece-wise smooth trajectories).
This implies
\cite[Theorems~II.26 and~II.28]{protter90} 
% \cite[Theorems~II.23 and~II.28]{protter04}
that $[M_\u,M_\v](t)$ is %piece-wise constant 
a pure jump process with jumps given by
\begin{align}
  \Delta[M_\u,M_\v](t)\,=\,\Delta M_\u(t)\Delta M_\v(t)
\,=\,e^{-(\widehat\mu(\u)+\widehat\mu(\v))t}\Delta\hcX_\u(t)\Delta\hcX_\v(t).
\end{align}
Similarly to the proof of Lemma~\ref{lma:martingale}, if the clock of a ball
at $\z$ rings, then by~\eqref{eq:hcX_jump} 
\begin{align}
\Delta\hcX_\u(t)\Delta\hcX_\v(t)\,=\,e^{i(\u+\v)\cdot\z}\widehat\varphi(\u)\widehat\varphi(\v).
\end{align}
Since the number of balls at $\z$ is $X_{\z}(t)$ and each rings with
intensity 1, we obtain from the above and \eqref{eq:Fmalthusian} that
\begin{align}
\frac{\d}{\d  t}\E\big[[M_\u,M_\v](t)\big]\,
&=\,\sum_{\z\in\Z^d}\E[X_{\z}(t)]e^{-(\widehat\mu(\u)+\widehat\mu(\v))t}e^{i(\u+\v)\cdot\z}\E\big[\widehat\varphi(\u)\widehat\varphi(\v)\big]
\notag\\ 
&=\,\E[\hcX_{\u+\v}(t)]e^{-(\widehat\mu(\u)+\widehat\mu(\v))t}\E\big[\widehat\varphi(\u)\widehat\varphi(\v)\big]
\notag\\ 
&=\,e^{(\widehat\mu(\u+\v)-\widehat\mu(\u)-\widehat\mu(\v))t}\E\big[\widehat\varphi(\u)\widehat\varphi(\v)\big].
\label{krk}
\end{align}
Integrating \eqref{krk} over the interval $[0,t]$,
recalling that
$[M_\u,M_\v](0)=M_\u(0)M_\v(0)=1$, 
and then using~\eqref{eq:XY-[XY]}
completes the proof.
\end{proof}

\begin{proof}[Proof of Proposition~\ref{prop:higher}]
Note that the complex conjugates of $\hcX_\u(t)$ and $\widehat\mu(\u)$ are
given by $\hcX_{-\u}(t)$ and $\widehat\mu(-\u)$, and consequently that 
\begin{align}
  |M_\u(t)|^2=M_\u(t)\overline{M_\u(t)}=M_\u(t)M_{-\u}(t).
\end{align}
By Lemma~\ref{lma:higher} we find that
\begin{align}\label{eq:lemma}
\E\big[|M_\u(t)|^2\big]
\,&=\,1+\E\big[|\widehat\varphi(\u)|^2\big]\int_0^t
 e^{(\widehat\mu(\0)-\widehat\mu(\u)-\widehat\mu(-\u))x}\,\d x
\notag\\
&=\,1+\E\big[|\widehat\varphi(\u)|^2\big]\int_0^t e^{(\lambda-2\Re\widehat\mu(\u))x}\,\d x.
\end{align}
Hence, for $\u\in\Tb^d$ such that $2\Re\widehat\mu(\u)>\lambda$ the 
complex-valued martingale
$(M_\u(t))_{t\ge0}$ is bounded in $L^2$. The existence of an almost sure and
$L^2$ limit $\Mx_\u$ is now a consequence of the 
martingale convergence theorem.
% (for complex-valued martingales). 
Moreover, as increments over disjoint time
intervals for square-integrable martingales are uncorrelated, we have for any
$s\ge t$ that 
\begin{align}
  \E\big[|M_\u(s)|^2\big]=\E\big[|M_\u(s)-M_\u(t)|^2\big]+\E\big[|M_\u(t)|^2\big]
\end{align}
and hence by~\eqref{eq:lemma}, since
$\E[|\widehat\varphi(\u)|^2]\le\E[\|\varphi\|^2]$, that 
\begin{align}
  \E\big[|M_\u(s)-M_\u(t)|^2\big]
\,=\,\E\big[|M_\u(s)|^2\big]-\E\big[|M_\u(t)|^2\big]
\,\le\,\E\big[\|\varphi\|^2\big]\int_t^s e^{(\lambda-2\Re\widehat\mu(\u))x}\,\d x.
\end{align}
Sending $s\to\infty$ thus yields~\eqref{eq:higher1}.

Arguing for~\eqref{eq:higher2} we first observe that
\begin{align}
  |M_\u(t)-M_\0(t)|^2
\,=\,M_\u(t)M_{-\u}(t)+M_\0(t)^2-M_\u(t)M_\0(t)-M_{-\u}(t)M_\0(t).
\end{align}
Hence, Lemma~\ref{lma:higher} gives that
\begin{align}
\E\bigsqpar{|M_\u(t)-M_\0(t)|^2}
\,&=\,\E\big[|\widehat\varphi(\u)|^2\big]
 \int_0^t e^{(\lambda-2\Re\widehat\mu(\u))x}\,\d x
 +\E\big[|\widehat\varphi(\0)|^2\big]\int_0^t e^{-\lambda x}\,\d x
\notag\\
&\qquad-2\Re\E\big[\widehat\varphi(\u)\widehat\varphi(\0)\big]
  \int_0^t e^{-\lambda x}\,\d x
\notag\\
&=\,\E\big[|\widehat\varphi(\u)|^2\big]\int_0^t
  \left(e^{(\lambda-2\Re\widehat\mu(\u))x}-e^{-\lambda x}\right)\,\d x
\notag\\
&\qquad+\E\big[|\widehat\varphi(\u)-\widehat\varphi(\0)|^2\big]\int_0^t e^{-\lambda x}\,\d x.
\end{align}
Since $\E[|\widehat\varphi(\u)|^2]\le\E[\|\varphi\|^2]$, estimating the
integrals leads to the upper bound
\begin{align}
\E\bigsqpar{|M_\u(t)-M_\0(t)|^2}
&\le
\E\big[\|\varphi\|^2\big]
\left(\frac{1}{2\Re\widehat\mu(\u)-\lambda}-\frac1\lambda\right)
+\frac1\lambda\E\big[|\widehat\varphi(\u)-\widehat\varphi(\0)|^2\big]
\notag\\
&\le
\E\big[\|\varphi\|^2\big]
\frac{2(\gl-\Re\hmu(\u))}{(2\Re\widehat\mu(\u)-\lambda)\gl}
+\frac1\lambda\E\big[|\widehat\varphi(\u)-\widehat\varphi(\0)|^2\big].  
\label{eq:h-bound1}
\end{align}
In order to obtain an upper bound of order $|\u|^2$ we first note that
\begin{align}
  |e^{i\u\cdot\z}-1|
%\,=\,\sqrt{2-2\cos(\u\cdot\z)}
\,\le\,|\u\cdot\z|
\le|\u||\z|,
\end{align}
using the mean-value theorem and Cauchy--Schwarz' inequality. 
Hence, recalling \eqref{fir},
\be\label{eq:h-bound2}
|\widehat\varphi(\u)-\widehat\varphi(\0)|\,\le\,\sum_{\z\in\Z^d}|e^{i\u\cdot\z}-1|\varphi(\z)\,\le\,|\u|\sum_{\z\in\Z^d}|\z|\varphi(\z)\,=\,|\u|\normm{1}{\varphi}.
\ee
Similarly we obtain, since $\mu(\z)=\E[\gf(\z)]$,
\be\label{eq:h-bound3}
\gl-\Re\widehat\mu(\u)=
\widehat\mu(\0)-\Re\widehat\mu(\u)
=\sum_{\z\in\Z^d}\bigpar{1-\cos(\u\cdot\z)}\E[\varphi(\z)]
\,\le\,\tfrac12|\u|^2\E[\normm{2}{\varphi}].
\ee
Hence, plugging~\eqref{eq:h-bound2} and~\eqref{eq:h-bound3}
into~\eqref{eq:h-bound1} leaves us with
\begin{align}
  \E\big[|M_\u(t)-M_\0(t)|^2\big]
\,\le\,\E\big[\|\varphi\|^2\big]\frac{|\u|^2\E[\normm{2}{\varphi}]}{(2\Re\widehat\mu(\u)-\lambda)\lambda}+\frac{1}{\lambda}|\u|^2\E\big[\normm{1}{\varphi}^2\big]
=O\bigpar{|\u|^2}
\end{align}
as required.
\end{proof}

\subsection{Bounds on the spatial displacement of balls}

We next consider the mean spatial distribution of $\cX(t)$.
%descendants of a single ball positioned at the origin. 
Recall that the expected total number of balls at time $t$
is $\E\norm{\cX(t)}=e^{\gl t}$ by \eqref{eq:malthusian}.
Define
\begin{align}
  \label{jc}
  p_\z(t)
  \, :=\,\frac{\E[X_{-\z}(t)]}{\E\|\cX(t)\|}
  \,=\,e^{-\lambda t}\E[X_{-\z}(t)],
\end{align}
the proportion of the expected number of balls at time $t$ 
that are expected to be at $-\z$, when starting (as always in this section)
from a single ball at the origin. Note that $p_\z(t)$ coincides with the expected contribution to the origin of a ball started at $\z$.
(The choice of $-\z$ is just for notational convenience in later sections, 
e.g.\ in~\eqref{EXzo} and~\eqref{qaa}.)
Note that, trivially by the definitions, for every $t\ge0$, we have $p_\z(t)\ge0$ and
\begin{align}\label{sump}
  \sum_{\z\in\Zd} p_\z(t)=1.
\end{align}

We shall next derive some key quantitative estimates that 
we shall use in later sections.

\begin{prop}\label{prop:key}
Assume that $\Phi$ is irreducible and satisfies $\E[\|\varphi\|^2]<\infty$,
$\E\big[\normm{1}{\varphi}^2\big]<\infty$ and
$\E\big[\normm{2}{\varphi}\big]<\infty$. Then, for $t\ge1$,
\begin{romenumerate}
\item\label{prop:key1}
  $\quad\ds \sup_{\z\in\Z^d}p_\z(t)=O(t^{-d/2})$,
\item\label{prop:key2}
  $\quad\ds \sum_{\z\in\Z^d}p_\z(t)^2=\Theta(t^{-d/2})$,
\item\label{prop:key3}
  $\quad\ds e^{-2\lambda t}\sum_{\z\in\Z^d}\E\left[\big(X_{\z}(t)-p_{-\z}(t)\|\cX(t)\|\big)^2\right]=O(t^{-(d+2)/2})$.
\end{romenumerate}
\end{prop}

\begin{proof}
We start with \ref{prop:key1}, and observe that by the definition \eqref{jc}
%of $p_\z(t)$
and~\eqref{eq:exp_inversion} we have 
\begin{align}
p_\z(t)
\,&=\,
%e^{-\lambda t}\E[X_{\z,\0}(t)]\,=\,
e^{-\lambda t}\E[X_{-\z}(t)]
%\notag\\
=\int_{\Tb^d}e^{i\u\cdot\z}e^{-\xpar{\lambda-\widehat\mu(\u)}t}\,\d\u
\,\le\,\int_{\Tb^d}e^{-\xpar{\lambda-\Re\widehat\mu(\u)}t}\,\d\u.
\label{eq:p_bound}\end{align}
To further bound the integral we have the following well-known standard
estimate, which highlights the importance of the irreducibility
assumption.
Note that a complementary upper bound (for all $\u$)
is given in \eqref{eq:h-bound3},
provided $\E[\normm{2}{\varphi}]<\infty$.

\begin{claim}\label{claim:irreducibility}
Assume that $\Phi$ is irreducible. 
Then, $\Re\widehat\mu(\u)<\lambda$ for all $\u\in\Tb^d\setminus\{\0\}$, and
there exists $c>0$ such that
\begin{align}
  \label{claim1}
\lambda-\Re\widehat\mu(\u)\,\ge\,c|\u|^2\quad\text{for all }|\u|\le c.
\end{align}
\end{claim}

\begin{proof}[Proof of Claim]
The statement is obtained by analyzing the identity
\begin{align}
  \lambda-\Re\widehat\mu(\u)
\,=\,\widehat\mu(\0)-\Re\widehat\mu(\u)
\,=\,\sum_{\z\in\Z^d}\big[1-\cos(\u\cdot\z)\big]\mu(\z).
\end{align}
We omit the details.
\end{proof}

Since $\Phi$ is irreducible, Claim~\ref{claim:irreducibility} gives a constant
$c>0$ such that~\eqref{claim1} holds. 
Let $K_c:=\{\u\in\Tb^d:|\u|\ge c\}$.
Since $K_c$ is compact, 
Claim~\ref{claim:irreducibility} and continuity of $\hmu(\u)$
also gives a 
constant $\gamma>0$ such that
$\lambda-\Re\widehat\mu(\u)\ge\gamma$ on $K_c$. Hence, together
with~\eqref{eq:p_bound}, 
\begin{align}
  \sup_{\z\in\Z^d}p_\z(t)
\,\le\,\int_{|\u|<c}e^{-c|\u|^2t}\,\d\u+\int_{K_c}e^{-\gamma t}\,\d\u
\,=\,O(t^{-d/2})+O(e^{-\gamma t}).
\end{align}
This proves part~\ref{prop:key1}.

The upper bound in \ref{prop:key2} is immediate from \ref{prop:key1} and 
\eqref{sump}.
% the fact that $\sum_{\z\in\Z^d}p_\z(t)=1$. 
Due to the identity in~\eqref{eq:p_bound} we may for a lower bound use Parseval's
formula together with \eqref{eq:h-bound3} to obtain that
\begin{align}
  \sum_{\z\in\Z^d}p_\z(t)^2
\,=\,\int_{\Tb^d}\lrabs{e^{-(\lambda-\widehat\mu(\u))t}}^2\,\d\u
\,\ge\,\int_{\Td}e^{-Ct|\u|^2 }\,\d\u
=t^{-d/2}\int_{(-\pi t\qq,\pi t\qq]^d}e^{-C|\u|^2 }\,\d\u,
\end{align}
where $C=\E[\normm2{\gf}]/2$, which for $t\ge1$ is bounded below by a constant times $t^{-d/2}$.

For \ref{prop:key3} we recall 
\eqref{M0} and
the definition
\eqref{jc}
of $p_\z(t)$,
by
which
\begin{align}
  X_{\z}(t)-p_{-\z}(t)\|\cX(t)\|\,=\,X_{\z}(t)-\E[X_{\z}(t)]M_\0(t).
\end{align}
Using the inversion formula~\eqref{eq:inversion}
and~\eqref{eq:exp_inversion} we find this equal to
\begin{align}
\int_{\Tb^d}e^{-i\u\cdot\z}\big[\hcX_\u(t)-e^{\widehat\mu(\u)t}M_\0(t)\big]\,\d\u\,=\,\int_{\Tb^d}e^{-i\u\cdot\z}e^{\widehat\mu(\u)t}\big[M_\u(t)-M_\0(t)\big]\,\d\u.
  \label{malin}
\end{align}
The right-hand side of \eqref{malin}
is the Fourier transform of a function on $\Tb^d$. Hence, by Parseval's formula, we obtain
\begin{align}
\sum_{\z\in\Z^d}\Bigabs{X_{\z}(t)-p_{-\z}(t)\|\cX(t)\|}^2
\,=\,\int_{\Tb^d}\Bigabs{e^{\widehat\mu(\u)t}\big[M_\u(t)-M_\0(t)\big]}^2\,\d\u.  
\end{align}
Taking expectation
yields
\be\label{eq:after_parseval}
\sum_{\z\in\Z^d}\E\Big[\big|X_{\z}(t)-p_{-\z}(t)\|\cX(t)\|\big|^2\Big]\,=\,\int_{\Tb^d}e^{2\Re\widehat\mu(\u)t}\E\big[|M_\u(t)-M_\0(t)|^2\big]\,\d\u.
\ee

Let $c>0$, $K_c$ and $\gamma>0$ be as above, so that
$\lambda-\Re\widehat\mu(\u)\ge c|\u|^2$ when $|\u|\le c$ and
$\lambda-\Re\widehat\mu(\u)\ge\gamma$ on 
on the complementary set
$K_c$.
We may
without loss of generality assume that $c>0$ was chosen so that also
$\lambda-\Re\widehat\mu(\u)\le\lambda/4$ for $|\u|\le c$ and that
$\gamma\le\lambda/4$. 

For $|\u|\le c$  we use \eqref{claim1} and \eqref{eq:higher2}, and find that
for some constant 
$\CC$, 
\begin{align}
  \label{ql}
e^{2\Re\widehat\mu(\u)t}\E\big[|M_\u(t)-M_\0(t)|^2\big]
\,\le\,\CCx e^{2(\lambda-c|\u|^2)t}|\u|^2,
\qquad |\u|\le c.
\CCdef{\Cql}
\end{align}
Next we observe that for all $\u$, \eqref{eq:lemma} implies that there
exists 
a constant $\CC$ such that
\begin{align}
\E\big[|M_\u(t)-M_\0(t)|^2\big]
\,&\le\,
2\E\big[|M_\u(t)|^2\big]+2\E\big[|M_\0(t)|^2\big]
%\\&
\le\,\CCx \Bigsqpar{1+\int_0^te^{(\lambda-2\Re\widehat\mu(\u))x}\,\d x}.
\label{b2}
\end{align}
By distinguishing between the cases 
$2\Re\widehat\mu(\u)\ge\frac54\lambda$ and
$2\Re\widehat\mu(\u)\le\frac54\lambda$,
we obtain from \eqref{b2}, rather crudely,
\begin{align}
\E\big[|M_\u(t)-M_\0(t)|^2\big]
&
\le\CC 
e^{\max\bigset{0,\tfrac32\lambda-2\Re\widehat\mu(\u)}t}.
\label{b2b}
\end{align}
Hence, on $K_c$,
\begin{align}
e^{2\Re\hmu(\u)t}\E\big[|M_\u(t)-M_\0(t)|^2\big]
&
\le\CCx 
e^{\max\bigset{2\Re\hmu(\u),\tfrac32\lambda}t}
\le\CCx 
e^{2(\gl-\gam)t},
\qquad \u\in K_c.
\CCdef{\Cqm}
\label{qm}
\end{align}
Combining \eqref{eq:after_parseval} with the estimates \eqref{ql}
  and \eqref{qm} yields
\begin{align}
e^{-2\lambda  t}
\sum_{\z\in\Z^d}\E\Big[\big|X_{\z}(t)-p_{-\z}(t)\|\cX(t)\|\big|^2\Big]
\,&\le\,\int_{|\u|<c}\Cql e^{-2c|\u|^2t}|\u|^2\,\d\u
+\int_{K_c}\Cqm e^{-2\gamma t}\,\d\u
\notag\\
&=\,O(t^{-(d+2)/2})+O(e^{-2\gamma t}).
\end{align}
This proves \ref{prop:key3} and thus completes the proof.
\end{proof}

\section{The monochromatic process}\label{sec:monochromatic}

We have defined the monochromatic process $\cYp(t)=(\Yp_\x(t))_{\x\in\Zd}$
as the process with independent Bernoulli distributed random initial values
$\Yp_\x(0)$
with parameter $p$.
Our main goal in the present section is to prove \refT{thm:Zd-limit}. A key step will be to derive a variance bound that will be central also later.
By translation invariance, $\Ypx(t)$ has the same distribution for all
$\x\in\Zd$, and we may consider only $\x=\0$.
%For convenience we write $\etax=\Yp_\x(0)$, so that $\bfeta=(\etax)_{\x\in\Z^d}$ has independent Bernoulli distributed entries.

We begin by introducing a useful representation. Let
$\bfeta=(\etaz)_{\z\in\Z^d}$ be a vector of independent 
Bernoulli distributed entries with parameter $p$.
(I.e., $\etaz\in\set{0,1}$ with $\Pr(\etaz=1)=p$.)
For each $\z\in\Zd$, let, as above, 
$\cXz(t)=(\Xzx(t))_{\x\in\Zd}$ be the process started with a single ball at $\z$,
and assume further that these processes are independent of each other and
of $\bfeta$.
Then we can construct the process $\cYp$ as
$\cYp(t)=\sumzzd \etaz \cXz(t)$, i.e., the process which for $\x\in\Z^d$ and $t\ge0$ is given by
\begin{align}\label{jac}
  \Ypx(t) = \sumzzd \etaz\Xzx(t).
\end{align}
(This is because in the monochromatic process there are no annihilations
and balls evolve independently.)

We next use this representation to prove the following key proposition.
\begin{prop}\label{propY}
Assume that $\Phi$ is irreducible and 
that $p\in(0,1]$.
%Then, the following hold.
\begin{romenumerate}
\item   \label{prop:E}
If $\gl=\E[\norm{\gf}]<\infty$, then
for every $t\ge0$,
\begin{align}\label{propE}
\E\bigsqpar{e^{-\lambda t}\Ypo(t)}
%\,=\, 
= p.
\end{align}

\item \label{prop:var}
If\/ $\E[\|\varphi\|^2]<\infty$,
$\E\big[\normm{1}{\varphi}^2\big]<\infty$ and
$\E\big[\normm{2}{\varphi}\big]<\infty$, then
for some constant $C=C(p,\Phi)$ and all $t\ge1$,
\begin{align}\label{propvar}
\Var\bigsqpar{e^{-\lambda t}\Ypo(t)}
\,=\, C\sum_{\z\in\Z^d}p_\z(t)^2+O(t^{-(d+1)/2})\,=\,\Theta(t^{-d/2}).
\end{align}
\end{romenumerate}
\end{prop}

\begin{proof}[Proof of \refP{propY}]\CCreset
\pfitemref{prop:E}
By \eqref{jac}, independence,
translation invariance, and \eqref{eq:malthusian},
using an interchange of order of summation and expectation, that is justified
since all variables are non-negative,
\begin{align}\label{jb}
 \E[\Ypo(t)]\,=\,p\sum_{\z\in\Z^d}\E[X_{\z,\0}(t)]
\,=\,p\sum_{\z\in\Z^d}\E[X_{\0,-\z}(t)]
\,=\,p\E[\|\cX_\0(t)\|]\,=\,pe^{\lambda t},
\end{align}
which yields \eqref{propE}.

\pfitemref{prop:var}
Under the assumption that $\E[\|\varphi\|^2]<\infty$, it follows by
Proposition~\ref{prop:higher} (with $\u=\0$) and \eqref{M0}
that for each $\z\in\Z^d$, 
the process $\{e^{-\lambda t}\|\cXz(t)\|:t\ge0\}$
is an $L^2$-bounded martingale. Hence, the limit 
\begin{align}
  \label{Wz}
W_\z:=\lim_{t\to\infty}e^{-\lambda t}\|\cXz(t)\|
\end{align}
exists almost surely and in $L^2$,
and \eqref{eq:malthusian} implies 
\begin{align}
  \label{EW}
\E [W_\z]=1.
\end{align} 
Note that by \eqref{M0},
\begin{align}  \label{W0}
W_\0 %:=\lim_{t\to\infty}e^{-\lambda t}\|\cXo(t)\|
=\lim_{t\to\infty}M_\0(t)
=\Mx_\0.
\end{align}

We decompose
$\Ypo(t)$ in the following manner, using \eqref{jac} and \eqref{sump}.
\begin{align}\label{eq:decompose}
e^{-\lambda t}\Ypo(t)-p
\,&=\,\sum_{\z\in\Z^d}\etaz e^{-\lambda t}\big(X_{\z,\0}(t)-p_\z(t)\|\cXz(t)\|\big)
\notag\\
&\qquad+\sum_{\z\in\Z^d}p_\z(t)\etaz\big(e^{-\lambda t}\|\cXz(t)\|-W_\z\big)
\notag\\
&\qquad+\sum_{\z\in\Z^d}p_\z(t)\big(\etaz W_\z-p\big).
\end{align}
We will prove below that the sums converge 
in $L^2$, and that their values as elements of $L^2$ are independent of the
order of summation,
so the decomposition is well-defined.
Moreover, although we don't really need this, the proof also shows that for
any given fixed order of summation (given by a fixed 
%but arbitrary
enumeration of $\Zd$), 
the sums converge \as{}

Denote the three sums on the right-hand side of~\eqref{eq:decompose} by
$\Sigma_1(t)$, $\Sigma_2(t)$ and $\Sigma_3(t)$. 
Note first that
by translation invariance, \eqref{jc} and \eqref{eq:malthusian},
  \begin{align}\label{EXzo}
\E[X_{\z,\0}(t)]
=\E[X_{\0,-\z}(t)]
= e^{\gl t}p_\z(t)
=p_\z(t)\E[\|\cX_\z(t)\|].
\end{align}
Hence, the terms in the first sum have zero mean.
The same holds for the terms in the second and third sums too 
by \eqref{eq:malthusian} and
\eqref{EW}.
It follows that
each of the three sums consists
of independent terms with zero mean;
hence a sufficient (and necessary) condition for the 
existence of the sum, 
in $L^2$ and almost surely (for any fixed order of summation),
is that the sum of the variance of the summands is
finite. We state this well-known standard result formally for easy reference.

\begin{claim}[Kolmogorov]\label{claim:kolmogorov}
Let $\xi_1,\xi_2,\ldots$ be independent zero mean random variables and let
$S_n$ denote the sum of the first $n$ of them. If\/
$\sum_{k=1}^\infty\Var(\xi_k)<\infty$, then
$S_\infty:=\lim_{n\to\infty}S_n$ exists almost surely and in $L^2$, and
\begin{align}
  \label{jd}
\Var\left(S_\infty\right)\,=\,\sum_{k=1}^\infty\Var(\xi_k).
\end{align}
\end{claim}

\begin{proof}[Proof of Claim]
For the existence of the limit, see e.g.~\cite[Lemma~6.5.2 and
Theorem~6.5.2]{gut13}. The formula then follows since
$\E[S_\infty^2]=\lim_{n\to\infty}\E[S_n^2]$. 
\end{proof}

We treat the three terms in \eqref{eq:decompose} separately and in order, 
obtaining estimates of the variance and at the same
time showing the existence of the sums $\gS_j(t)$ in $L^2$ and a.s.

%The terms of $\Sigma_1(t)$ are all independent and have zero mean. So,
%applying 
First, Claim~\ref{claim:kolmogorov} and translation invariance show (assuming it is finite)
that
\begin{align}
\EVar{\Sigma_1(t)}
\,&=\,p\,e^{-2\lambda t}\sum_{\z\in\Z^d}
 \E\left[\big(X_{\z,\0}(t)-p_\z(t)\|\cXz(t)\|\big)^2\right].
\notag\\
\,&=\,p\,e^{-2\lambda t}\sum_{\z\in\Z^d}
 \E\left[\big(X_{\0,-\z}(t)-p_\z(t)\|\cXo(t)\|\big)^2\right].
\label{eq:Sigma1}
\end{align}
By Proposition~\ref{prop:key}\ref{prop:key3} the right-hand side is indeed finite,
so $\Sigma_1(t)$ is well-defined and~\eqref{eq:Sigma1} justified. By the
same proposition, we find that, for some $\CC<\infty$,
\begin{align}\label{gS1}
\EVar{\Sigma_1(t)}
= \EEVar{\Sigma_1(t)}
\le \CCx t^{-(d+2)/2}.
\CCdef\CgSi  
\end{align}

Similarly,
Claim~\ref{claim:kolmogorov} yields
\begin{align}
  \label{je}
\EVar{\Sigma_2(t)}
\,=\,p\sum_{\z\in\Z^d}p_\z(t)^2
 \E\left[\big(e^{-\lambda t}\|\cXz(t)\|-W_\z\big)^2\right],
\end{align}
where we note that by translation invariance,
\eqref{M0}, \eqref{W0},
and \refP{prop:higher},
% since $\widehat\mu(\0)=\lambda$, 
\begin{align}
\E\left[\big(e^{-\lambda t}\|\cXz(t)\|-W_\z\big)^2\right]
= \E\left[\big(e^{-\lambda t}\|\cXo(t)\|-W_\0\big)^2\right]
=  \E\left[(M_\0(t)-\Mx_\0)^2\right]
\le \CC e^{-\gl t}.
\label{jf}
\end{align}
Recalling \eqref{sump}, we see from \eqref{jf} that the sum in \eqref{je}
converges, so $\gS_2(t)$ is well-defined, and furthermore
\begin{align}
  \label{gS2}
\EVar{\Sigma_2(t)}
= \EEVar{\Sigma_2(t)}
\le \CCx e^{-\gl t}.
\end{align}

Finally, 
by Proposition~\ref{prop:higher} the variables $W_\z\eqd W_\0=\Mx_\0$ 
exist in $L^2$, so Claim~\ref{claim:kolmogorov}  gives
\begin{equation}\label{eq:Sigma3}
\EVar{\Sigma_3(t)}
\,=\,\sum_{\z\in\Z^d}p_\z(t)^2\E\big[(\eta_\z W_\z-p)^2\big]
\,=\,\bigpar{p\E[W_\0^2]-p^2}\sum_{\z\in\Z^d}p_\z(t)^2.
\end{equation}
Note that $\E[W_\0^2]\ge \E [W_\0]^2=1$.
Furthermore, equality would imply $W_\0=1$ a.s., and thus the martingale 
$e^{-\lambda t}\|\cXz(t)\|$ would be constant, i.e., 
$\norm{\cXz(t)}=e^{\gl t}$ a.s.\ for every $t\ge0$, which is absurd.
Hence, $\E[W_\0^2]>1$, and thus
${p\E[W_\0^2]-p^2}>0$ for all $p\in(0,1]$.
Proposition~\ref{prop:key}\ref{prop:key2} and \eqref{eq:Sigma3} thus show that 
\begin{align}\label{gS3}
\EVar{\Sigma_3(t)}
= \EEVar{\Sigma_3(t)}
=\Theta(t^{-d/2}),
\qquad t\ge1.  
\end{align}

Examining the variance estimates
\eqref{gS1}, \eqref{gS2} and \eqref{gS3},
we conclude that $\gS_3(t)$ has a variance of larger order than the
other two sums.
We conclude that the variance of $\Sigma_3(t)$ is the dominating term in
the variance of $\Ypo(t)$; more precisely, \eqref{eq:decompose} and
Minkowski's inequality imply, using \eqref{gS1} and \eqref{gS2},
\begin{align}
\Bigabs{\bigpar{ \Var \sqpar{e^{-\lambda t}\Ypo(t)}}\qq 
- \bigpar{\E\sqpar{\gS_3(t)^2}}\qq}
\le
\bigpar{\E\sqpar{\gS_1(t)^2}}\qq
+ \bigpar{\E\sqpar{\gS_2(t)^2}}\qq
= O\bigpar{t^{-(d+2)/4}}.
\end{align}
The first equality in \eqref{propvar} now follows by \eqref{eq:Sigma3} 
and \eqref{gS3} 
(with $C=p\E[W_\0^2]-p^2>0$). 
The second equality follows for large $t$ by
Proposition~\ref{prop:key}\ref{prop:key2}; it trivially extends to all
$t\ge1$
since %the variance
$\Var \sqpar{e^{-\lambda t}\Ypo(t)}$, as a simple consequence of \eqref{jac}, is
bounded below by some positive number for every bounded interval $[1,T]$.
%This completes the proof of Proposition~\ref{propY}. 
\end{proof}

\begin{remark}
In fact, it is easy to see, e.g.\ as a consequence of
\cite[III.4.(5)]{athney72}, that
$\E[W_\z^2]=1+\E[\norm{\gf}^2]/\gl$, and thus 
$C=p\E[\norm{\gf}^2]/\gl+p-p^2$.
Note also that the bounds \eqref{gS1} and \eqref{gS2} are uniform in $p$.
\end{remark}

\begin{proof}[Proof of Theorem~\ref{thm:Zd-limit}]
The case $p=0$ is trivial, so we assume $p>0$.
By translation invariance, we may assume $\x=0$.
\refP{propY} immediately yields $L^2$ convergence in \eqref{tYlimit}, so
it  remains to establish almost sure convergence.
We shall show, for every fixed $\delta>0$, that
\begin{align}\label{delta-as}
e^{-\lambda\delta n}\Ypo(\delta n)\asto p
\qquad \text{as \ntoo}.
\end{align}
Since
  $\Ypo(t)$ is non-decreasing in $t$, this implies that, \as,
  \begin{align}\label{Fe}
e^{-\lambda\delta}p
\le\liminf_{t\to\infty}e^{-\lambda t}\Ypo(t)\le\limsup_{t\to\infty}e^{-\lambda t}\Ypo(t)
\le e^{\lambda\delta}p.    
  \end{align}
Hence, \as{} \eqref{Fe} holds for
all rational $\delta>0$, which implies 
$\lim_{t\to\infty}e^{-\lambda t}\Ypo(t)=p$.

Thus, fix $\gd>0$.
In order to show \eqref{delta-as}, we again use the decomposition
\eqref{eq:decompose} and show that
$\gS_1(\gd n)\asto0$,
$\gS_2(\gd n)\asto0$,
and $\gS_3(\gd n)\asto0$ as \ntoo.

First, \eqref{gS1} shows that
%we have 
\begin{align}\label{gs1sum}
\E\bigg[\sum_{n=1}^\infty\Sigma_1(\delta
  n)^2\bigg]=\sum_{n=1}^\infty\E\big[\Sigma_1(\delta n)^2\big]\le
\CgSi\sum_{n=1}^\infty(\delta n)^{-(d+2)/2}
<\infty.
\end{align}
%for some constant $C$, which hence is finite. 
In particular, \as, $\sum_{n=1}^\infty\Sigma_1(\delta n)^2<\infty$
and thus $\lim_{n\to\infty}\Sigma_1(\delta n)=0$. 

Similarly, \eqref{gS2} implies that \as{}
$\lim_{n\to\infty}\Sigma_2(\delta n)=0$. 

To complete the proof of the theorem, it remains to show that also \as{}
$\Sigma_3(\delta n)\to0$ as $n\to\infty$. For
$d\ge3$ this follows from \eqref{gS3}, under the assumption that $\E[\|\varphi\|^2]<\infty$, 
just like for $\Sigma_1(t)$ and $\Sigma_2(t)$. 
For $d=1,2$ we need to argue differently, which requires a stronger
moment condition.

We will appeal to a theorem of Pruitt~\cite{pruitt66}. Note that
$\Sigma_3(\delta n)$ is of the form $\sum_{k\ge1}a_{n,k}X_k$, where the
random variables $X_k:=\eta_{\z_k}W_{\z_k}-p$ 
(for some arbitrary enumeration $(\z_k)_k$ of $\Z^d$)
are i.i.d.\ with mean zero and the coefficients 
$a_{n,k}:=p_{\z_k}(\delta n)$ 
are non-negative and satisfy, by Proposition~\ref{prop:key}\ref{prop:key1},
\begin{equation}
%\lim_{n\to\infty}a_{n,k}=0,\quad
\sum_{k\ge1}a_{n,k}=1\quad\text{and}\quad\max_{k\ge1}a_{n,k}=O(n^{-d/2}).
\end{equation}
Let $r=\max\{1+2/d,2\}$, which for $d=1$ gives $r=3$ and for $d\ge2$ gives $r=2$.
Note that the assumption $\E[\|\varphi\|^r]<\infty$
implies $\E[W_\z^r]<\infty$, and hence that
$\E[|X_k|^r]<\infty$; see
\cite[Theorems 1 and 3]{bindon75}.\footnote{Alternatively, see 
\cite[Corollary to Theorem 5]{bindon74}, 
applied to the Galton--Watson process $\norm{\cX_\z(n)}$,
and 
\cite[Corollary~III.6.1]{athney72}.} %gives $\E\norm{\cX(t)}^r<\infty$
Consequently, if $\E[\|\varphi\|^r]<\infty$, then \cite[Theorem~2]{pruitt66} gives that \as{} $\lim_{n\to\infty}\Sigma_3(\delta n)=0$.

We have shown that each of the three sums $\gS_j(t)$ on the right-hand side
of~\eqref{eq:decompose} \as{} tends to 0 for $t=n\gd\to\infty$,
which as said above yields \eqref{delta-as} and
completes the proof of Theorem~\ref{thm:Zd-limit}. 
\end{proof}

\begin{remark}\label{prutt}
The work of Pruitt~\cite{pruitt66} was brought to our attention by Luca Avena and Conrado da Costa. Our previous proof (for the cases $d=1,2$) was based on Rosenthal's inequality (see \cite[Theorem~3.9.1]{gut13}) and required the stronger conditions $\E[\|\varphi\|^{4+\eps}]<\infty$ when $d=1$ and $\E[\|\varphi\|^{2+\eps}]<\infty$ when $d=2$, for some $\eps>0$.
\end{remark}

\section{The two-type annihilating process}
\label{sec:existence}

We now, finally, turn to the two-colour competition process. Foremost, we shall prove that the process is well-defined in the generality that it is studied in this paper, and along the way establish some properties that we will need for the proof of our main theorem.

For monochromatic initial configurations the process is, as already mentioned, well-defined as there are simply no interactions between different balls. For non-monochromatic initial configurations consisting of finitely many non-zero elements it is straightforward to construct the annihilating process, since \as{} only finitely many nucleation events occur in finite time, and no two balls nucleate simultaneously, so annihilations can be carried out in chronological order. For similar reasons the process is well-defined for initial configurations in which at least one of the colours appear in finite numbers; we refer to such configurations as \emph{quasi-monochromatic}. Also in this setting there are \as{} at most finitely many nucleation events occurring in finite time that may result in an annihilation of balls, and the annihilations can thus be carried out as before.

It is, however, less obvious that for arbitrary initial configurations the
process exists as we have described it. Since there is no `first' event of
annihilation, an attempt to determine whether a potential annihilation takes
place or not could (in principle) result in the tracing of an infinite
sequence of 
potential annihilations backwards in time. This
\emph{should} not be the case. However, in order to avoid this problem, we
shall take a limiting approach where we define the annihilating process for
a general initial configuration as a limit of the process for a sequence of
finite initial configurations. In order to do so properly, we 
shall need to detail further how the process is constructed. Throughout this
section we shall limit our attention to initial configuration in
$\{-1,0,1\}^{\Z^d}$. 
For configurations,
we use the product order on $\bbZ^{\bbZ^d}$, and write thus
$\bgz\le\bgz'$ for configurations $\bgz=(\gz_\bx)_\bx$ and $\bgz'=(\gz'_\bx)_\bx$
if and only if $\gz_\bx\le \gz'_\bx$ for every $\bx\in\bbZ^d$.

For the reader who prefers to postpone the details of this section and
proceed to the proof of our main theorem, we remark that Lemma~\ref{Ltory}
below will be used to prove fixation for $p\neq1/2$, and
Lemmas~\ref{Lexp}--\ref{Llim} will be used in the proof of non-fixation at
$p=1/2$. In addition, these lemmas are used in this section to justify our
definition of the annihilating process for arbitrary initial configurations.
Some of the lemmas will be proven first for the \qmc{} case, and in the
present section used only for that case; at the end of the section we
extend the proofs to the general case.

\subsection{A technical digression on the construction of the process}\label{SStechnical}

%At times we shall require some care in the construction of the 
%competition process.
For the remainder of this paper we make (without loss of generality)
the following assumptions.
We label each ball (regardless of its colour)
by a finite string $(\bz,i_1,i_2,\dots,i_m)$ with
$\z\in\Zd$, $m\ge0$ and
$i_j\in\bbN$, such that
the ball initially at $\bz$ (if any) is labelled by $(\bz)$, and if a ball 
has label $(\bz,i_1,i_2,\dots,i_m)$, then its children are labelled by
$(\bz,i_1,i_2,\dots,i_m,i)$ for $i=1,2,\dots$ (in some fixed order).
This gives each ball a unique label.
Furthermore, we assume that we have a Poisson clock for each possible label;
these clocks are independent of each other and of the initial configuration.
Moreover, each clock is equipped with one realization of the random
offspring configuration $\gf$ for each ring of the clock.
We now define the process with each ball using the corresponding clock and
the copies of $\gf$ provided by that clock.
(Ticks and tocks of unused clocks are ignored.) Furthermore, when a ball annihilates another,
and there are several balls at that site that may be chosen for
annihilation, we chose 
the one that comes first 
according to some fixed rule, for example the ball
with smallest label in lexicographic order (using an arbitrary but fixed
order on $\bbZ^d$). 
Note that all randomness in the process $\cZ(t)$ now lies in the clocks and the
initial configuration; $\cZ(t)$ is a deterministic function of these. Moreover, all clocks may be assumed to start ticking at the dawn of time, and are thus completely independent of the initial configuration.

At occasions we will want to emphasise or compare different initial
configurations, and thus write $\cZ(t,\bgz)$ for the state at time $t\ge0$
of the process with initial configuration $\bgz$. Since clocks are
independent of the presence and colour of the balls in the initial
configuration, this yields a coupling $\{\cZ(t,\bgz)\}$ of the processes
for all possible initial configurations. Due to the independence between the
clocks and the initial configuration, we shall throughout this section
consider deterministic initial configurations; analogous statements for
random initial configurations are obtained through conditioning.

%The conservative process is defined in the same way, using the same clocks; when two balls merge to a purple ball, one of them is newborn and the other one is chosen among the existing balls at the site by the same rule as for annihilations in the competition process; furthermore, the purple ball inherits, for definiteness, the label (and thus the clock) of the ball that existed at the site before merging.

\subsection{A conservative version of the annihilating system}
\label{sec:coupling}

We introduce a conservative version of the process,
previously explored in~\cite{ahlgrijanmor19}. 
In this process, 
red and blue balls branch and get offspring as in the competition
process described above,
but when a red and a blue ball meet, instead of annihilating, the two balls
merge to form a purple ball. Each purple ball in the system continues to
branch independently and according to the same rule as red and blue. (For definiteness, purple balls inherit the label (and thus the clock) of the older of the two balls involved in the merging.)
Purple balls, however, do not interact with other balls. 
Consequently, we recover the competition process by ignoring all purple balls.

For quasi-monochromatic initial configurations the conservative process is well-defined
for the same reasons the annihilating process is well-defined. Let
$R_\bx(t)$,
$B_\x(t)$, and $P_\x(t)$
be the numbers of red, blue, and purple balls, respectively,
at site $\x\in\Zd$ at time $t$ in the above conservative process,
and let
$\cR(t):=(R_\bx(t))_{\bx\in\bbZ^d}$,
$\cB(t):=(B_\bx(t))_{\bx\in\bbZ^d}$,
and $\cP(t):=(P_\bx(t))_{\bx\in\bbZ^d}$ 
be the corresponding vectors.
Then the competition process is given by $\cZ(t)=\cR(t)-\cB(t)$. 
Furthermore, we use the standard notation
$x_+:=\max(x,0)$ and $x_-:=\max(-x,0)$ for real $x$, 
and extend this component-wise to vectors $\bgz=(\zeta_\x)_{\x\in\Zd}$.
Then, in particular, $\cR(t)=\cZ(t)_+$ and $\cB(t)=\cZ(t)_-$.

The crucial facts about this conservative process are stated in the following
lemma. The lemma will in the coming sections allow us to apply the previous results for the monochromatic process in order to prove Theorem~\ref{thm:fixnonfix}.

\begin{lemma}\label{Ltory}
The conservative process is well-defined for any initial configuration
$\bgz$.
 The process
$\cR(t)+\cP(t)$ is an instance of the monochromatic process
started with $\bgz_+$,
and 
$\cB(t)+\cP(t)$ is an instance of the monochromatic process started
with $\bgz_-$.  
Furthermore, for all $t\ge0$,
\begin{align}\label{ja}
  \cZ(t,\bgz)=\cR(t)-\cB(t)
=\bigsqpar{\cR(t)+\cP(t)} - \bigsqpar{\cB(t)+\cP(t)}.
\end{align}
\end{lemma}
\begin{proof}[Proof when $\bgz$ is \qmc]
  Immediate from the definitions.
\end{proof}

As said above, the general case will be treated at the end of the section.
Note that the two monochromatic processes 
$\cR(t)+\cP(t)$ and $\cB(t)+\cP(t)$ 
are \emph{not} independent and 
that they are equal to the processes $\cZ(t,\bgz_+)$ and $\cZ(t,\bgz_-)$ in
distribution, but not necessarily point-wise.

\subsection{Comparison of initial configurations}

We next state two lemmas that will help to compare versions of the process
with different initial configurations; 
the lemmas are proved for \qmc{}
initial configurations in the present subsection, and in genral at the end
of the section.

We first note that the expected
configuration at a given time is a linear function of the initial
configuration, and state this formally for the number of balls at the
origin. 

\begin{lemma}\label{Lexp}
  For any (deterministic)  initial configuration 
$\bgz$ we have
  \begin{align}\label{qaa}
    \E [\Zo(t,\bgz)]= e^{\gl t} \sum_{\bz\in\bbZ^d}p_\bz(t)\gz_\bz.
  \end{align}
\end{lemma}

\begin{proof}[Proof when $\bgz$ is \qmc]
  First, consider the monochromatic case; say $\bgz\ge\bo$, so all balls are
  red. Then all balls evolve independently,
so $\Zo(t,\bgz)=\sum_{\bz\in\bbZ^d} X_{\bz,\bo}(t) \gz_\bz$, \cf{} \eqref{jac},
  and \eqref{qaa} follows by
  linearity (since all terms are non-negative) and \eqref{EXzo}.

  In general, we introduce purple balls and 
use \refL{Ltory}.
Then, by \eqref{ja} and using \eqref{qaa} for each of the monochromatic processes
 $\cR(t)+\cP(t)$ and $\cB(t)+\cP(t)$,
  \begin{align}
 \E [\Zo(t,\bgz)] = \E\sqpar{R_\0(t)+P_\0(t)} - \E\sqpar{B_\0(t)+P_\0(t)}
    =e^{\gl t}\sum_{\bz:\gz_\bz=1}p_\bz(t) - e^{\gl t}\sum_{\bz:\gz_\bz=-1}p_\bz(t),
  \end{align}
 which is well-defined since both sums are finite, and \eqref{qaa} follows.
\end{proof}

Our next lemma states that
the process is monotone in the initial configuration.
Recall that each ball is represented by a unique label of the form
$(\z,i_1,i_2,\ldots,i_m)$. 
We may identify balls and their labels when convenient; thus when we talk about 
a given ball in the competition process, we mean a ball with a given label.
We refer to a label as \emph{active} at time $t$
if the corresponding ball exists in the process $\cZ(t,\bgz)$. Let
$\ActR_\z(t,\bgz)$ denote the set of labels corresponding to red balls at
$\z$ that are active at time $t$, and let $\ActB_\z(t,\bgz)$ denote ditto
for labels corresponding to blue balls.

\begin{lemma}\label{Lmon}
Let $\bgz$ and $\bgz'$ be two  configurations with $\bgz\le\bgz'$.
Then, \as{} for all $t\ge0$ and $\z\in\Z^d$ we have
\begin{equation}\label{active}
\ActR_\z(t,\bgz)\subseteq\ActR_\z(t,\bgz')\quad\text{and}\quad\ActB_\z(t,\bgz')\subseteq\ActB_\z(t,\bgz).
\end{equation}
In particular, \as{}  $\cZ(t,\bgz)\le\cZ(t,\bgz')$ for all $t\ge0$.
\end{lemma}

\begin{proof}[Proof when $\bgz$ is \qmc]
The inequality $\bgz\le\bgz'$ means that
  every red ball in $\bgz$ exists (with the same label and colour) 
also in   ${\bgz'}$,
  and every blue ball in $\bgz'$ exists also in $\bgz$.
  (There may also be further
  red balls in $\bgz'$ and blue balls in ${\bgz}$.)
  We claim that this holds at all later times $t\ge0$ too, and thus
  that~\eqref{active} holds. 
  In fact, since balls with the same label obey the same clock,
  the property \eqref{active}
is preserved at each nucleation (including accompanying
  annihilations, since they follow a fixed order), 
as is easily verified. (If
  a nucleation results in a red ball in the process starting from $\bgz$,
  then the nucleation will result in a red ball with the same label
  also in the process started from $\bgz'$.)
  If $\bgz$ and $\bgz'$ are finite, the result now follows by induction over
  the number of nucleations, which then is finite in every finite interval.
  
  In the general \qmc{} case, with initial configurations that may be
  infinite, the conclusion will follow just the same, since there are at
  most finitely many nucleations that may result in an annihilation in
  finite time intervals. To see this, note that the process obtained by
  suppressing all annihilations is simply the monochromatic process (the
  process in which each entry $\zeta_\z$ of $\bgz$ is replaced by
  $|\zeta_\z|$). In this process each ball initially present will result in
  at most finitely many descendants in finite time (\cf{} 
  \refR{Rfinite}). In particular, in each of 
 $\cZ(t,\bgz)$ and $\cZ(t,\bgz')$ there are at
  most finitely many balls born with one of the colours in every finite time window,
  and thus at most a finite number of potential annihilations.
\end{proof}

\subsection{General initial configurations}

We next show that  the competition process 
with arbitrary initial configurations can be defined as the limit of
the process for
finite or quasi-monochromatic configurations; we shall also see that this 
limit
indeed satisfies the verbal description of the annihilating process given
above.

We let $|\cdot|$ denote the $\ell^\infty$-norm on $\R^d$ and set
\begin{equation}
B(\org,r):=[-r,r]^d=\{\x:|\x|\le r\},\quad\text{for }r\ge0.
\end{equation}
Given a configuration $\bgz$ and an integer $r\ge0$, 
let $\QLE{\bgz}{r}$ 
denote the restriction of $\bgz$ to $B(\0,r)$,
i.e., $\QLE{\gz}{r}_\bz:=\gz_\bz\cdot\indic{|\bz|\le r}$.
We define also the modifications
$\bgz\ppr,\bgz\mmr$; these are  equal to $\bgz$ in 
$B(\0,r)$, but for $\x\notin B(\0.r)$ we set
\begin{align}
%\gz\pqr_\x:=|\gz_\x|,&&
%\gz\mr_\x:=-|\gz_\x|,&&
\gz\ppr_\x:=1\quad\text{and}\quad %&&
\gz\mmr_\x:=-1. & 
\end{align}
This means that outside $B(\0,r)$ we put one ball at each site, 
red  in $\bgz\ppr$ and blue in $\bgz\mmr$. 
Note that $\QLE{\bgz}r$ is finite and that
$\bgz\ppr$ and $\bgz\mmr$ are quasi-monochromatic.

The inequalities
\begin{align}
  \label{penta}
\cZ(t,\bgz\mmr)
%\le \cZ(t,\bgz\mr)
\le\cZ(t,\QLE{\bgz}{r})
%\le\cZ(t,\bgz\pqr)
\le\cZ(t,\bgz\ppr)
\end{align}
hold for $t=0$ 
by definition,
and \as{} for every $t\ge0$ by \refL{Lmon}.
Similarly we have
\begin{align}\label{quarto}
  \cZ(t,\bgz\mmr)\le \cZ(t,\bgz^{-,r+1}) \quad\text{and}\quad
  \cZ(t,\bgz\ppr)\ge \cZ(t,\bgz^{+,r+1})
\end{align}
for every $t\ge0$.

\begin{lemma}\label{Llim}
For any initial configuration $\bgz\in\{-1,0,1\}^{\Z^d}$, $\z\in\Z^d$ and $T<\infty$ there exists \as{} a (random) $L<\infty$ such that for all $r\ge L$ and $t\in[0,T]$
\begin{equation}\label{acteq}
\ActR_\z(t,\bgz\ppr)=\ActR_\z(t,\bgz\mmr)\quad\text{and}\quad\ActB_\z(t,\bgz\ppr)=\ActB_\z(t,\bgz\mmr),
\end{equation}
so that, in particular, 
\begin{align}
  \label{zeq}
Z_\z(t,\bgz\ppr)=Z_\z(t,\bgz\mmr)
\end{align}
for all $t\in[0,T]$ and $r\ge L$.
\end{lemma}

\begin{proof}
Fix $T>0$. By Lemma~\ref{Lmon} we have \as{} that for all $\z\in\Z^d$, $t\in[0,T]$ and $r\ge0$
\begin{equation}\label{eq:Lbounda}
\ActR_\z(t,\bgz\mmr)\subseteq\ActR_\z(t,\bgz\ppr)\quad\text{and}\quad\ActB_\z(t,\bgz\ppr)\subseteq\ActB_\z(t,\bgz\mmr).
\end{equation}
%Consequently, \as{} $\cZ^-(t,\bgz)\le\cZ^+(t,\bgz)$ for all $t\in[0,T]$.
For the reverse inclusion, we consider only the case $\z=\org$ for
notational convenience; the general case is analogous. Thus, let
$\mathcal{E}_r(T)$ 
denote the event that~\eqref{eq:Lbounda} holds (for all $\bz$)
but that for
some $t\in[0,T]$ equation~\eqref{acteq} fails for $\z=\org$. 
%Since no two nucleations occur simultaneously,
% and since balls in each nucleation have the same colour, 
Then, on the event $\mathcal{E}_r(T)$ we have,
for some $t\in[0,T]$,
$Z_\org(t,\bgz\mmr)+1\le Z_\org(t,\bgz\ppr)$ 
and thus
\begin{equation}\label{eq:Lbound}
\cZ(t,\bgz\mmr)+\delta_{\x,\org}\le\cZ(t,\bgz\ppr).
\end{equation}
An induction argument, as in the proof of Lemma~\ref{Lmon}, shows that~\eqref{eq:Lbound} implies
\begin{equation}\label{we}
\cZ(T,\bgz\mmr)+\delta_{\x,\org}\le\cZ(T,\bgz\ppr).
\end{equation}
In particular, $Z_\org(T,\bgz\mmr)+1\le Z_\org(T,\bgz\ppr)$. 
Consequently, Markov's inequality and~\eqref{eq:Lbounda}, together with Lemma~\ref{Lexp}, yield
\begin{equation}\label{tho}
\Pr\big(\mathcal{E}_r(T)\big)\le
%\Pr\big(Z_\org(T,\bgz\ppr)-Z_\org(T,\bgz\mmr)\ge1\big)\le
\E\big[Z_\org(T,\bgz\ppr)-Z_\org(T,\bgz\mmr)\big]\le2e^{\lambda t}\sum_{|\z|>r}p_\z(T),
\end{equation}
which tends to zero as $r\to\infty$. 
Moreover, \eqref{quarto} implies that $\cE_r\supseteq\cE_{r+1}$.
Hence, the event that \eqref{eq:Lbounda} holds but \eqref{acteq} fails for
arbitrarily large $r$ is 
$\cE(T):=\limsup_{r\to\infty}\cE_r(T)=\bigcap_{r\ge1}\cE_r(T)$, and \eqref{tho}
yields 
$\Pr\big(\mathcal{E}(T)\big)=0$ as required.
\end{proof}

\begin{lemma}\label{LX}
As $r\to\infty$,
the three processes
$\cZ(t,\QLE{\bgz}{r})$,
$\cZ(t,\bgz\mmr)$,
%$\cZ(t,\bgz\mr)$,
%$\cZ(t,\bgz\pqr)$,
and
$\cZ(t,\bgz\ppr)$
\as{} all converge in the product space
$D[0,\infty)^{\Z^d}$ to a common limiting process
%\begin{equation}
$\cZ^\ast(t,\bgz)
%:=\lim_{r\to\infty}\cZ(t,\QLE{\bgz}{r}).
$.
%\end{equation}
Moreover, the convergence is in the strong sense that for every $\z$ and
$T<\infty$, there exists $L=L(\z,T)$
 such that the processes at $\z$ all are equal to
the limit $Z^\ast_\z(t,\bgz)$ for all $t\in[0,T]$ and  $r\ge L$.  
\end{lemma}

\begin{proof}
\refL{Llim} shows that
  given $\z$ and $T$, there exists $L$ such that \eqref{zeq} holds for $r\ge
  L$ and all $t\le T$.
By \eqref{quarto}, this implies
$Z_\z(t,\bgz\ppr)=Z_\z(t,\bgz\mmr)=Z_\z^{+,L}=Z_\z^{-,L}$
for $r\ge L$. The result follows by \eqref{penta}.
\end{proof}

We will henceforth take $\cZx(t,\bgz)$ as the formal definition of
$\cZ(t,\bgz)$ for every initial configuration $\bgz$;
in other words, we define
\begin{equation}\label{lim}
\cZ(t,\bgz):=\cZ^\ast(t,\bgz)=\lim_{r\to\infty}\cZ(t,\QLE{\bgz}{r}),
\end{equation}
recalling that by \refL{LX}, we could as well use 
$\cZ(t,\bgz\ppr)$ or $\cZ(t,\bgz\mmr)$
%any of the four other modifications 
in \eqref{lim}.
 Note that if $\bgz$ is
quasi-monochromatic, then
\refL{Lmon} implies that \as{}
\begin{align}
  \label{penta=}
\cZ(t,\bgz\mmr)
\le\cZ(t,\bgz)
\le\cZ(t,\bgz\ppr)
\end{align}
for all $t$ and $r$, 
and thus
$\cZx(t,\bgz)=\cZ(t,\bgz)$ by \refL{LX}, so there is no inconsistency. 

We verify next
that this process indeed behaves as the competition process that we have
described verbally in 
the introduction and 
Section~\ref{SStechnical}. 

\begin{lemma}\label{Ldef}
The process $\cZ(t,\bgz)$, defined by \eqref{lim},
%as the limiting process $\cZx(t,\bgz)$, 
has the following almost sure properties: For each site $\x\in\Z^d$ it holds that
\begin{romenumerate}
\item \label{Ldefa}
balls at $\x$ have a distinct labels and produce labeled offspring according to the corresponding clocks;
\item \label{Ldefb}
balls arriving at $\x$ are those born in nucleations at (possibly different) sites and annihilate according to the predefined rule;
\item \label{Ldefc}
there are no additional balls or annihilations occurring at $\x$.
\end{romenumerate}
\end{lemma}

\begin{proof}
Fix $T<\infty$, $\x\in\Z^d$ and a configuration
$\bgz\in\{-1,0,1\}^{\Z^d}$. First, note that properties
\ref{Ldefa}--\ref{Ldefc} hold, by construction, for any quasi-monochromatic
initial configuration. Next, by Lemma~\ref{Llim} there exists \as{}
$r<\infty$ so that the labels active at $\x$ up to time $T$ coincide for
$\cZx(t,\bgz)$ and $\cZ(t,\bgz^{+,r})$. So, \ref{Ldefa} follows since it
holds for $\cZ(t,\bgz^{+,r})$.

Let $\bgz^{++}$ denote the monochromatic configuration consisting of one red ball at each site, and let $N_\x$ denote the set of locations where there is a nucleation sending a ball to $\x$ at some time $t\le T$ in the process $\cZ(t,\bgz^{++})$. Note that for every quasi-monochromatic configuration $\bgz'$, the locations at which there is a nucleation in $\cZ(t,\bgz')$ sending a ball to $\x$ (immediately annihilated or not) is a subset of $N_\x$.

Since $\cZ(T,\bgz^{++})$ is \as{} locally finite, the set $N_\x$ is \as{}
finite. According to Lemma~\ref{Llim} there exists \as{} $r<\infty$ so that
the labels active up to time $T$ at each $\z\in N_\x\cup\{\x\}$ coincide for
$\cZx(t,\bgz)$ and $\cZ(t,\bgz^{+,r})$. Consequently, the nucleations in
$\cZx(t,\bgz)$ resulting in a ball att $\x$ are the same as those in
$\cZ(t,\bgz^{+,r})$, and these are the only balls that appear at $\x$. It
follows that the annihilations at $\x$ caused by these nucleations are the
same for $\cZx(t,\bgz)$ and $\cZ(t,\bgz^{+,r})$, and that there are no
annihilations occurring other than these. This proves 
\ref{Ldefb} and \ref{Ldefc}.
\end{proof}

\begin{remark}\label{Rk}
Let us note that the limiting process $\cZx(t,\bgz)$ does not depend on the
order in which the initial 
configuration is being `revealed' in the limiting
procedure. Indeed, fix any enumeration of $\Z^d$ and let $\bgz^{(k)}$ 
denote the configuration whose coordinates equal $\zeta_\z$
for the first $k$ entries $\z$ of the enumeration and 0 otherwise.
Then, for every $r$ we may choose
$K<\infty$ so that $\bgz\mmr\le\bgz^{(k)}\le\bgz\ppr$
for all $k\ge K$. By Lemma~\ref{Lmon} it follows that for all $k\ge K$ and
$t\ge0$
\begin{equation}
\cZ(t,\bgz\mmr)\le\cZ(t,\bgz^{(k)})\le\cZ(t,\bgz\ppr).
\end{equation}
By Lemma~\ref{LX} we conclude that  also
$\cZ(t,\bgz^{(k)})$ converges to the  limiting process $\cZx(t,\bgz)$
as $k\to\infty$. 
(The same argument applies to the corresponding versions of 
$\cZ(t,\bgz\ppr)$ and $\cZ(t,\bgz\mmr)$.)
In particular, the distribution of the limiting process is
invariant with respect to translations.
\end{remark}

\subsection{Epilogue}

Having properly defined the process we may now extend the lemmas proven
above for quasi-monochromatic 
initial configurations to arbitrary configurations. 
\begin{proof}[Proof of \refL{Ltory}, general case]
 As in the proof of Lemma~\ref{Ldef},
the annihilations occurring at $\x$ up to time $T$ are the same for
$\cZx(t,\bgz)$ and $\cZ(t,\bgz^{+,r})$ when $r$ is large. Consequently, the
process consisting of purple balls is well-defined for arbitrary initial
configurations, as the limit as $r\to\infty$ of the purple process for
$\bgz\ppr$. This defines the conservative process for any $\bgz$, and it is
immediate that the claims in the lemma hold, since they hold for each
$\bgz\ppr$.
\end{proof}

\begin{proof}  [Proof of \refL{Lexp}, general case]
The proof above now applies to arbitrary $\bgz$.
\end{proof}

\begin{proof}[Proof of \refL{Lmon}, general case]
If $\bgz\le\bgz'$, then $\bgz\ppr\le\xpar{\bgz'}\ppr$ for every $r$, and
the already proven case applies. The result follows by \refLs{Llim} and 
\ref{LX}.
\end{proof}

\section{Fixation versus non-fixation}\label{sec:fixnonfix}

The goal of this section is to prove Theorem~\ref{thm:fixnonfix}. The first
part of the theorem, stating that the competing urn scheme started from an
unbalanced initial Bernoulli colouring will eventually fixate at red, is an
immediate consequence of Theorem~\ref{thm:Zd-limit} together with the
conservative coupling, see \refSS{ssec:fix}; 
the coupling decomposes the competition process into
the difference between two monochromatic processes, one of which is larger
by a factor. 

The second part of the theorem, stating that the competing urn scheme started from a balanced initial Bernoulli colouring does not fixate, will require more work. Let us first present a brief sketch of the proof. We start with the intuition that the state of the origin at time $t=1$ is unlikely to dictate the state of the origin at time $t\gg1$. Taking this intuition to its logical conclusion we should be able to choose a fast growing sequence of times $t_1,t_2,\dots$ such that the state of the origin at time $t_n$ is approximately independent from its states at times $t_1,\dots, t_{n-1}$.  One would therefore expect the origin to be red for infinitely many of the times $t_n$ and blue for infinitely many of the times $t_n$, which would complete the proof.

In order to make this rigorous we show that the state of the origin at time $t=1$ mostly depends on the descendants of balls which start near the origin. On the contrary, at time $t\gg1$ the state at the origin depends on descendants of balls from a much larger region, while balls originating near the origin contribute little. This will allow us to define a growing sequence of scales $r_1,r_2,\dots$ such that the state of the origin at time $t_n$ may be well approximated by considering \emph{only} the descendants of balls initially in the annulus $[-r_{n+1},r_{n+1}]^d\setminus [-r_n,r_n]^d$.  Since these annuli are disjoint, this will allow us to `decouple' the state of the origin at times $t_1,t_2,\dots$.

We implement this approach in Section~\ref{sec:decoupling}.  An important ingredient will be to understand the likely order of magnitude of the number of balls at the origin.  While the order of magnitude does not grows as fast as $e^{\lambda t}$ (as in the case $p\neq1/2$), it is still
likely to be at the order of its standard deviation, which is
$t^{-d/4}e^{\lambda t}$. This will be obtained via a second moment approach,
resting on the quantitative bounds for the monochromatic process
obtained in Proposition~\ref{propY}; see Section~\ref{sec:second} for details.

\subsection{Fixation for $p\neq1/2$}\label{ssec:fix}

We first consider the annihilating process $\cZ(t)$ in the unbalanced setting, that is, starting from a biased Bernoulli colouring, and prove fixation.

\begin{proof}[Proof of Theorem~\ref{thm:fixnonfix}\ref{thm:fixnonfixa}] 
Using \refL{Ltory}, we decompose $\cZ(t)$ into the difference of two monochromatic processes as follows:
\begin{equation}\label{kaa}
\cZ(t)=[\cR(t)+\cP(t)]-[\cB(t)+\cP(t)].
\end{equation}
Next, fix $\eps>0$ such that $2p-1>3\eps$. By
Theorem~\ref{thm:Zd-limit} there exists \as\ a (random) finite $T_0$ such
that, for all $t\ge T_0$, 
\begin{align}
  e^{-\lambda t}\big(R_\0(t)+P_\0(t)\big)>p-\eps\quad\text{and}\quad 
e^{-\lambda t}\big(B_\0(t)+P_\0(t)\big)<1-p+\eps,
\end{align}
and hence \eqref{kaa} yields $Z_\0(t)>\eps e^{\lambda t}>0$, as required. That is, the origin fixates \as{} to red, and by
translation invariance, every site \as{} fixates.
\end{proof}

\subsection{A second moment analysis for $p=1/2$}\label{sec:second}

In this subsection we consider the annihilating process $\cZ(t)$ in the
balanced setting, that is, 
starting from a symmetric Bernoulli colouring 
in which each site is given a
ball whose colour is determined by a fair coin flip.
We aim to prove the following bound on deviations of
$\cZ(t)$.

\begin{prop}\label{prop:deviation}
Consider the competing urn scheme $\cZ(t)$ in the balanced setting,
and assume that $\E\big[\normm{1}{\varphi}^2\big]<\infty$,
$\E\big[\normm{2}{\varphi}\big]<\infty$
and $\E[\|\varphi\|^2]<\infty$ holds.
%Let $p=1/2$.
Then there exists a constant $c>0$ such that for all $t>1/c$ we have
\begin{align}
  \Pr\big(e^{-\lambda t}\Zo(t)>ct^{-d/4}\big)>c.
\end{align}
\end{prop}

To show this, we shall use the following generic lemma, which is a conditional
version of the Paley--Zygmund inequality.

\begin{lma}\label{lma:sec_mom}
Let $X$ be a random variable and $\Fc$ a sub-\gsf{} (on some probability space). Suppose that $\E[X^2]\le K$ and $\Pr\big(\E[X\mid\Fc]\ge1/K\big)\ge 1/K$ for some constant $K$. Then
\begin{align}
  \Pr\big(X>1/(2K)\big) \ge 1/(4K^5).
\end{align}
\end{lma}

\begin{proof}
Let $F$ be the event that $\E[X\mid\Fc]\ge1/K$ and let $E$ be the event that
both 
$\E[X\mid\Fc]\ge1/K$ and $X>1/(2K)$. Cauchy--Schwartz gives
\begin{align}
  \label{aa}
\Ex[X \ind_E]\, \le\, \|X\|_2 \|\ind_E\|_2\, \le\, K^{1/2}\pr{E}^{1/2}\, .
\end{align}
On the other hand we have that
\begin{equation}
\Ex[X\ind_F]\, \ge \, \frac{\pr{F}}{K}\,
\end{equation}
and, since $X\le 1/2K$ on $F\setminus E$,
\begin{equation}
\Ex[X \ind_{F\setminus E}]\, \le\, \frac{\pr{F\setminus E}}{2K}
\, \le \, \frac{\pr{F}}{2K}.
\end{equation}
Thus, since $\Pr(F)\ge 1/K$ by assumption,
\begin{equation}
\Ex[X\ind_{E}]\, \ge\, \frac{\pr{F}}{2K}\, \ge\, \frac{1}{2K^2}\,,
\end{equation}
which combined with \eqref{aa} gives
\begin{equation}
K^{1/2}\pr{E}^{1/2}\, \ge\, \frac{1}{2K^2}\, .
\end{equation}
It follows that $\pr{E}\ge 1/4K^5$, and the result follows.
\end{proof}

We also need an estimate of the variance.
\begin{lemma}
  \label{LVarZ}
Under the conditions of Proposition~\ref{prop:deviation}, we have for all $t\ge1$
  \begin{align}
    \Var \sqpar{e^{-\gl t}\Zo(t)}
    = O\bigpar{t^{-d/2}}.
    %\le C t^{d/2}
  \end{align}
\end{lemma}

\begin{proof}
  We use the conservative process with purple balls and \eqref{ja},
  recalling that both
  $R_\0(t)+P_\0(t)$ and   $B_\0(t)+P_\0(t)$ have the same distribution as
  $Y^p_\0(t)$ with $p=1/2$.
  Hence, by \refP{propY}\ref{prop:var},
  \begin{align}
    \Var\sqpar{\Zo(t)}
    \le 4 \Var\bigsqpar{\Yhalf_\0(t)} =O\bigpar{t^{-d/2}e^{2\gl t}},
  \end{align}
  as required.
\end{proof}

We are now in position to proceed with the proof of the proposition.

\begin{proof}[Proof of Proposition~\ref{prop:deviation}]
  Let $\bgz$ be a random symmetric Bernoulli colouring and let
  $\cF$ be the \gsf{} generated by $\bgz$.
  Since the clocks are independent of $\bgz$, \refL{Lexp} yields
  \begin{align}\label{xb}
 \E\bigl[e^{-\gl t}\Zo(t)\mid\cF\bigr]=\sum_{\bz\in\bbZ^d}p_\bz(t)\zeta_\bz =: S(t).
  \end{align}
  $S(t)$ is a sum of independent random variables with mean 0, so using \refP{prop:key},
\begin{equation}\label{eq:var_ref}
  \Var \bigsqpar{S(t)}=\sum_{\bz\in\Z^d}p_\z(t)^2
  =\Theta\bigpar{t^{-d/2}}.
\end{equation}
%In the sequel, let
%$$
%s_M(t):=\bigg(\sum_{\z\in\Z^d}p_\z(t)^2\bigg)^{1/2}.
%$$
We next claim that, as $t\to\infty$, we have
\begin{equation}\label{eq:M_normal}
S(t)/\sqrt{\Var\sqpar{S(t)}}\to N(0,1)\quad\text{in distribution}.
\end{equation}
To see this, we use the central limit theorem with
the Lyapounov condition that
\begin{align}\label{lya}
  \beta(r,t):=
  \bigpar{\Var [S(t)]}^{-r/2}\sum_{\bz\in\bbZ^d}\E\bigsqpar{|p_\bz(t)\zeta_\bz|^r}
  =o(1)
\quad  \text{as } t\to\infty,
\end{align}
for some $r>2$; see \eg{} \cite[Theorem 7.2.2 and 7.2.4]{gut13}.
(The central limit theorem is usually stated for finite sums, but extends
immediately to $L^2$-convergent sums by truncation and the Cram\'er--Slutsky
theorem.)
We verify the Lyapounov condition \eqref{lya} with
$r=3$.
Then, by Proposition~\ref{prop:key}\ref{prop:key1}
\begin{align}\label{xa}
  \sum_{\bz\in\bbZ^d}\E\bigsqpar{|p_\bz(t)\zeta_\bz|^3}
  =   \sum_{\bz\in\bbZ^d}p_\bz(t)^3
\le\sup_{\z}p_\z(t)^{2}=O(t^{-d}).  
\end{align}
Hence, \eqref{eq:var_ref} and \eqref{xa} yield
$\beta(3,t)=O(t^{-d/4})$, which verifies \eqref{lya}, so
\eqref{eq:M_normal} holds.

To complete the proof,
let
$X:=e^{-\gl t}\Zo(t)/\sqrt{\Var[ S(t)]}$.
Recall that $\Zo(t)$ has mean zero, so we may from
\refL{LVarZ} and~\eqref{eq:var_ref} deduce that
$\E\bigsqpar{e^{-2\lambda t}\Zo(t)^2}\le K\Var\sqpar{S(t)}$
for some constant $K$; in other words,
$\E \sqpar{X^2}\le K$.
Increasing $K$ if
necessary, it follows from \eqref{xb} and~\eqref{eq:M_normal} that 
\begin{align}
  \Pr\bigpar{\E\sqpar{X\mid \cF}\ge 1/K}=
\Pr\big(S(t)\ge \sqrt{\Var\sqpar{S(t)}}/K\big)\ge 1/K
\end{align}
for all large $t$. Lemma~\ref{lma:sec_mom} therefore shows that
\begin{align}
  \Pr\Big(e^{-\lambda t}\Zo(t)> \sqrt{\Var[S(t)]}/(2 K)\Big)\ge1/(4K^{5}).
\end{align}
Since $\sqrt{\Var[S(t)]}=\Theta(t^{-d/4})$ by~\eqref{eq:var_ref},
the proof is complete.
\end{proof}

\subsection{A decoupling argument and conclusion of the proof}\label{sec:decoupling}

We now complete the proof of part \ref{thm:fixnonfixb} of \refT{thm:fixnonfix}.  We began this section with an overview of the proof, including the idea that we would consider a sequence of times $t_1,t_2,\dots$ and scales $r_1,r_2,\dots$ such that the state of the origin at time $t_n$ mostly depends on the descendants of balls which start in the annullus $B(\0,r_{n+1})\setminus B(\0,r_n)$.  We now implement this idea rigorously.  In addition to the sequences of times and scales we define an auxiliary sequence $(C_i)_{i\ge1}$ which controls the contribution at the origin of balls descending from within a growing sequence of regions.
%on the state of the origin at time $t_{n+1}$.

Since we shall need to quantify the contribution coming from different locations we introduce some further notation.
For a configuration $\bgz$, 
recall that $\QLE{\bgz}{r}$ denotes the restrictions of $\bgz$ to $B(\0,r)$;
we similarly
let  $\QGT{\bgz}{r}$
denote the restriction of $\bgz$ to the complement of $\Br$,
i.e.,
$\QGT{\gz}{r}_\bz=\gz_\bz\cdot\indic{|\bz|>r}$.
Finally, let $\QGLE{\bgz}{r}{r'}:=\QLE{(\QGT{\bgz}{r})}{r'}$. 
%Note that $\bgz^{-,r}\le\bgz^{\le r}\le\bgz^{+,r}$, so by Lemma~\ref{Lmon}
%it follows \as{} that 
%\begin{equation}\label{617}
%\cZ(t,\bgz^{-,r})\le\cZ(t,\bgz^{\le r})\le\cZ(t,\bgz^{+,r})\quad\text{for
%all }t\ge0. 
%\end{equation}

We first bound the number of balls at the origin that originate from $B(\0,r)$. 

\begin{lma}\label{lma:resampling}
For every $r\ge1$ and $\delta>0$ there exists $C>0$ such that for all
$t\ge1$ and every $\bgz\in\set{-1,0,1}^{\bbZ^d}$ we have
\begin{align}
\Pr\Bigpar{\bigabs{\Zo(t,\bgz)-\Zo(t,\QGT{\bgz}{r})}>Ct^{-d/2}e^{\lambda t}}
<\delta.
\end{align}
\end{lma}

\begin{proof}
  Recall that in $\QGT{\bgz}r$, all $\bgz_\bx$ with $|\bx|\le r$ have been
  reset to 0.
  Fix $r$ and define $\bgzp$ by instead letting
  $\bgzp_\bx:=1$ when $|\bx|\le r$, and as before $\bgzp_\bx=\bgz_\bx$  otherwise.
Then $\bgzp\ge\bgz$ and $\bgzp\ge\QGTz{r}$, and thus by \refL{Lmon},
$\Zo(t,\bgzp)\ge \Zo(t,\bgz)$ and $\Zo(t,\bgzp)\ge \Zo(t,\QGTz{r})$.
Consequently, the triangle inequality and \refL{Lexp} (applied four times) give
\begin{align}
  \E\bigabs{\Zo(t,\bgz)-\Zo(t,\QGT{\bgz}{r})}
& \, \le\,
  \E\bigl[\Zo(t,\bgzp)-\Zo(t,\bgz)\bigr]
  +
  \E\bigl[\Zo(t,\bgzp)-\Zo(t,\QGT{\bgz}{r})\bigr]
  \notag\\&
  \,\le\, 3 e^{\gl t} \sum_{\bz\in B(\0,r)} p_\bz(t)  % 3 sic!
  \,\le\, 3(2r+1)^de^{\gl t}\sup_\bz p_\z(t).
\end{align}
The result follows by Proposition \ref{prop:key}\ref{prop:key1}
and Markov's inequality.
\end{proof}

We may now, finally, complete the proof of our main theorem.

\begin{proof}[Proof of \refT{thm:fixnonfix}\ref{thm:fixnonfixb}]
Throughout the proof $\bgz$ denotes a random symmetric Bernoulli colouring, and $c>0$ is the constant from Proposition~\ref{prop:deviation}.  Let $t_0=1/c$ and let $r_0=0$.  We now define the sequences $(r_i)_{i\ge 1}$, $(C_i)_{i\ge 1}$ and $(t_i)_{i\ge 1}$ sequentially.
For $i\ge1$, choose
\begin{romenumerate}
\item\label{f1}
  $r_i>r_{i-1}$, using \refL{LX} and \eqref{lim},
%\eqref{penta}, \eqref{penta=}, and Lemma~\ref{Llim},
  such that
  \begin{align}
    \label{f1a}
    \Pr\Bigpar{\Zo(t_{i-1},\bgz)\neq \Zo(t_{i-1},\QLE{\bgz}{r_i})}\le 2^{-i};
  \end{align}

\item\label{f2}
  $C_i>0$, using Lemma~\ref{lma:resampling},
  such that for every $t\ge 1$
and $r\ge r_i$, 
\begin{align}\label{f2a}
\Pr\Bigpar{\bigabs{\Zo(t,\QLEz{r}) -\Zo(t,\QGLE{\bgz}{r_i}r)}>C_it^{-d/2}e^{\lambda t}}
\le 2^{-i}.
  \end{align}
\item\label{f3}
  $t_i>t_{i-1}+1$ such that $t_i^{d/4}>3c^{-1}C_i$, 
so that, by Proposition~\ref{prop:deviation}, we have
    \begin{align}
      \label{f3a}
\Pr\Bigpar{\Zo(t_i,\bgz)>3C_it_i^{-d/2}e^{\lambda t_i}} \ge c.
    \end{align}
  \end{romenumerate}

In particular,  \eqref{f2a} yields
  \begin{align}\label{f2b}
\Pr\Bigpar{\bigabs{\Zo(t_i,\QLEz{r_{i+1}})-\Zo(t_i,\bgzi)}>C_it_i^{-d/2}e^{-\lambda t_i}}
\le 2^{-i}.
  \end{align}
  Hence, using also \eqref{f1a},
  \begin{align}\label{f1b}
\Pr\Bigpar{\bigabs{\Zo(t_i,\bgz)-\Zo(t_i,\bgzi)}>C_it_i^{-d/2}e^{-\lambda t_i}}
    \le 2^{-i}+2^{-i-1}
    \le 2^{1-i}.
  \end{align}
 Thus, \eqref{f3a} implies 
  \begin{align}\label{f3b}
    \Pr\Bigpar{\Zo(t_i,\bgzi)>2C_it_i^{-d/2}e^{\lambda t_i}} \ge c- 2^{1-i}.
  \end{align}
  
Next, we define two `failure' events, of which (at least) one must occur for the
origin to be blue at all large times.

\begin{itemize}
\item
Let $F_1$ be the event that \emph{for infinitely many} $i\ge1$ we have
\begin{equation}\label{eq:F1}
\bigabs{\Zo(t_i,\bgz)-\Zo(t_i,\bgzi)}>C_it_i^{-d/2}e^{-\lambda t_i}.
\end{equation}
\item  Let $F_2$ be the event that \emph{for at most finitely many} $i\ge1$ we have
\begin{equation}\label{eq:F2}
\Zo(t_i,\bgzi)>2C_it_i^{-d/2}e^{\lambda t_i}.
\end{equation}
\end{itemize}
By \eqref{f1b},
% the recursive choices of the scales $r_i$, $C_i$ and $t_i$
it follows that~\eqref{eq:F1}  occurs with probability
at most $2^{1-i}$. Hence, by the Borel--Cantelli lemma, we have that
$\Pr(F_1)=0$. Moreover, by \eqref{f3b},
\eqref{eq:F2} occurs with probability at least $c/2$ for large $i$.
Note that the processes $\cZ(t,\QGLE{\bgz}{r_i}{r_{i+1}})$, for $i\ge1$,
are mutually independent by our construction.
Hence the events in \eqref{eq:F2} are independent,
and thus the other Borel--Cantelli lemma implies that $\Pr(F_2)=0$. 

Let $I_1$ and $I_2$ denote the sets of $i$'s for
which~\eqref{eq:F1} and~\eqref{eq:F2} occur, respectively.
We have shown that $I:=I_2\setminus I_1$ is infinite almost surely.
To complete the proof, we note that for each $i\in I$ we have
\begin{align*}
  \Zo(t_i,\bgz)&\ge \Zo(t_i,\bgzi)-C_it_i^{-d/2}e^{\lambda t_i}
  % \\&
  >C_it_i^{-d/2}e^{\lambda t_i}
             >0.
\end{align*}
Hence, there are \as{} arbitrarily large times $t$ such that $\Zo(t,\bgz)>0$ and
thus $\bo$ is red.
By symmetry there are \as{} also arbitrarily large $t$ with $\bo$ being blue.
This completes the proof of part~\ref{thm:fixnonfixb} of
Theorem~\ref{thm:fixnonfix}.
\end{proof}

\section{Existence of the density}\label{sec:density}

 In this section $\cZ(t)$ will describe the evolution of the system
starting from the $p$-random Bernoulli colouring, where $p\in[0,1]$ is arbitrary. We aim to show that the
density of red sites, as defined in~\eqref{eq:density}, is indeed \as{}
well-defined for all $t\ge0$.

\begin{thm}\label{Tdensity}
Assume $\E\,\norm{\gf}<\infty$.
Then, for the competing urn scheme on $\Z^d$ starting from a $p$-random
Bernoulli colouring, 
almost surely,
the density $\rho(t)$ of red urns, as defined in~\eqref{eq:density},
exists for all $t\ge0$ and 
\begin{align}\label{sylta}
  \rho(t)=\Pr\big(Z_\0(t)>0\big).
\end{align}
\end{thm}

We begin with a lemma.

\begin{lemma}\label{L0}
Assume $\E\,\norm{\gf}<\infty$. For every $T<\infty$ and $\eps>0$ there
exists
$\gd>0$ such that for any $p\in[0,1]$ and interval $I\subseteq[0,T]$ of length at most $\gd$, we have
%, then, for any $\bgz\in\set{-1,0,1}^{\bbZ^d}$,
\begin{align}\label{l0}
  \Pr\bigl(Z_\0(t) \text{\rm\ is {not} constant for $t\in I$}\bigr)
<\eps.
\end{align}  
\end{lemma}
\begin{proof}
Using \refL{Ltory} (and replacing $\eps$ by $\eps/2$), we see that it
suffices to prove the corresponding result for the monochromatic process $\cY^p(t)$.
 %(with an arbitrary value of $p$).
Let $[a,b]\subseteq [0,T]$.
Since the monochromatic process is (weakly) increasing at each site,
we obtain using
Markov's inequality together with \eqref{propE} and its proof,
\begin{align}
  \Pr\bigl(Y_\0^p(t) \text{\rm\ is {not} constant for $t\in [a,b]$}\bigr)
&=  \Pr\bigl(Y_\0^p(b)>Y_\0^p(a)\bigr)
\le\E[Y_\0^p(b) -Y_\0^p(a)]
\notag\\&
= p\bigl(e^{\gl b}-e^{\gl a}\bigr)
\le (b-a) \gl e^{\gl T}.
\end{align}  
The result follows by taking $\gd$ small enough.
\end{proof}

\begin{proof}[Proof of \refT{Tdensity}]
First consider a fixed $t\ge0$. We use a standard type of argument.
That the limit in \eqref{eq:density}  exists, almost surely, 
for a fixed $t\ge0$ 
is a consequence of 
translation invariance and
the (multivariate) ergodic theorem
(see \eg{} \cite[Theorem 10.12]{kallenberg02}).
Using the construction of $\cZ(t)$ in \refSS{SStechnical}, 
$\cZ(t)$ is a measurable deterministic function of 
the clocks and the initial colouring. Furthermore, 
by first considering monochromatic processes and then using \refL{Ltory},
we see that
changing a finite
number of the clocks and initial colours can only affect $Z_\x(t)$ for
finitely many $\x$, \as{}, which will not change the limit
\eqref{eq:density}. Thus $\rho(t)$ is measurable with respect to the
corresponding tail \gsf, and the Kolmogorov 0--1 law implies that
$\rho(t)$ is \as{} equal to a deterministic constant.
Finally, by taking expectations in \eqref{eq:density} and using 
the bounded convergence theorem, \as,
\begin{equation}\label{Edensity}
\rho(t)=
\E \rho(t)
=
\lim_{n\to\infty}\frac{1}{(2n+1)^d}\sum_{\z\in[-n,n]^d}
%\Pr\xpar{\z\text{ red at time }t}
\Pr\bigpar{Z_\z(t)>0}
= \Pr\bigpar{Z_\0(t)>0}.
\end{equation}
This establishes \eqref{sylta} for a fixed $t\ge0$.

We next show how to extend this equality to all $t\ge0$
simultaneously. 
Define the upper and lower densities $\rhou(t)$ and $\rhol(t)$ as in 
\eqref{eq:density} but using $\limsup$ and $\liminf$, respectively.
These are thus always defined, and \as{} equal to each other and given by~\eqref{sylta}.

Given an interval $I$, define similarly $\rhou_+(I)$ as the upper density of
sites that are red for \emph{some} $t\in I$, and  $\rhol_-(I)$ as the 
lower density 
of points that are red for \emph{all} $t\in I$. The argument just given for
$\rho(t)$ shows also that these densities \as{} exist and are equal to the
corresponding probabilities at $\0$.
Let $T<\infty$ and $\eps>0$, and let $\gd$ be as in \refL{L0}.
Then \eqref{l0} implies that for any fixed
interval $I\subseteq[0,T]$ of length at most $\gd$ we have
\as
\begin{align}\label{eva}
\rhou_+(I)-\rhol_-(I)<\eps
\end{align}
Furthermore, for all $t\in I$,
\begin{align}\label{anna}
  \rhol_-(I)\le \rhol(t)\le\rhou(t)\le\rhou_+(I),
\end{align}
and thus by \eqref{eva}, \as,
\begin{align}\label{axel}
  \sup_{t\in I}\bigpar{\rhou(t)-\rhol(t)} \le\rhou_+(I)-\rhol_-(I)<\eps.
\end{align}
By covering $[0,T]$ by a finite number of intervals of length at most
$\delta$ we conclude that \as
\begin{align}
  \label{theodor}
\sup_{t\in [0,T]}\bigpar{\rhou(t)-\rhol(t)} <\eps,
\end{align}
and sending $\eps\to0$ and $T\to\infty$ 
shows that \as\
$\rhou(t)=\rhol(t)$ for all $t$ simultaneously.

Furthermore, write for convenience 
$f(t):=\Pr\bigpar{Z_\0(t)>0}$, so $\rho(t)=f(t)$ \as{}
for each fixed $t$.
Fix again an interval $I$ as above.
Then \eqref{l0} implies that for any $s,t\in I$,
$|f(s)-f(t)|<\eps$.
Fix $s\in I$.
Since \eqref{anna} and \eqref{eva} imply that \as
\begin{align}\label{gerion}
 | \rhou_+(I)-f(s)|
=  | \rhou_+(I)-\rhou(s)|
\le \rhou_+(I)-\rhol_-(I) <\eps,
\end{align}
it follows that \as, %for all $t\in I$ simultaneously,
\begin{align}\label{nix}
\sup_{t\in I} | \rhou_+(I)-f(t)|
\le | \rhou_+(I)-f(s)| + \sup_{t\in I} | f(s)-f(t)|
<2\eps
\end{align}
and thus, using \eqref{anna} and \eqref{eva} again, \as
\begin{align}\label{nax}
\sup_{t\in I} | \rhou(t)-f(t)|
< 2\eps + |\rhou(s)-\rhou_+(I)|
<2\eps+ \bigpar{\rhou_+(I)-\rhol_-(I)} 
<3\eps.
\end{align}
By covering $[0,T]$ by a finite number of intervals of length at most
$\delta$ we conclude that \as{} the same holds for $I$ replaced by $[0,T]$, and then
sending $\eps\to0$ and $T\to\infty$ shows that \as{} $\rhou(t)=f(t)$ for
all $t\ge0$. 
Hence, \as, $\rho(t)=\rhou(t)=f(t)$ for all $t$ simultaneously.
\end{proof}

Together with Theorem~\ref{thm:fixnonfix} it follows that for $p>\frac12$ \as{} $\rho(t)\to1$ as $t\to\infty$.

\section{Dealing with death}\label{s:death}

As we have defined our process, at each ring of a clock, the corresponding ball produces offspring according to $\Phi$, and remains itself where it was. Consequently, in the monochromatic version of our process, once a ball is born, it remains at its position at all future times. We shall in this section describe briefly how the results obtained for this process can be extended to allow balls to die (disappear) as they reproduce.
(Recall \refR{r:death}.)
%This can be achieved by considering the more general version of the process where particles die (disappear) once they nucleate (produce offspring).
This is obviously more general, since we can let the parent be replaced by a copy of itself.
In particular, this extension allows us to consider models where the balls move around at random, such as the standard (continuous time, discrete space) branching random walk where particles perform independent simple symmetric random walks, and in each step, with some probability, split in two or more independent copies.

So, consider the model in which particles have an exponentially distributed life time, at the end of which they are removed and replaced by a configuration $\varphi$, shifted to the position of the particle, drawn from $\Phi$. We assume, as before, that $\Phi$ is an irreducible probability measure on finite non-negative (but not necessarily non-empty) configurations on $\Z^d$, satisfying $1<\E[\norm{\gf}]<\infty$. Under the condition that $\E[\norm{\gf}]>1$, then the total number of balls $\norm{\cX(t)}$, when starting from a single ball at the origin, is a supercritical branching process.\footnote{Note that we above have assumed, implicitly, that $\Phi$ is supported on nonempty configurations, as the contrary would simply correspond to a rescaling of time. This is no longer assumed here, resulting in the possible extinction of the process evolving from a single ball. Extinction will, of course, not be possible when starting from an infinite starting configuration.} We outline below how our arguments may be adapted to cover this more general family of processes.
(As before, we assume that $\E[\norm{\gf}^{4+\eps}]<\infty$ and that~\eqref{eq:displacement} holds where appropriate.)

There are four places at which our arguments need modification. First, in
Section~\ref{sec:oneball}, we need to compensate for the death of
particles. Write $\varphi'$ for the \emph{change} caused as a clock rings at
the origin, and by $\mu'$ its expectation. Then
$\varphi'=\varphi-\delta_{\x,\0}$ and $\mu'(\x)=\mu(\x)-\delta_{\x,\0}$. 
Similarly, redefine $\gl:= \E[\norm{\gf}]-1>0$.
By
replacing $\varphi$ and $\mu$ by $\varphi'$ and $\mu'$
throughout Section~\ref{sec:oneball}, then all
results continue to hold for the more general family of processes.
(Note, in particular, how the
expression $\lambda-\Re\widehat\mu(\u)$ is unaffected by these changes.) 

Secondly, in Section~\ref{sec:monochromatic}, when $\cY^p(t)$ is no longer
non-decreasing, we need an argument to deduce~\eqref{Fe}
from~\eqref{delta-as}. A simple large deviation estimate will suffice, since
if $Y_\0^p(t)$ changes significantly during a short time span, then a
greater than expected number of clock rings must have occurred. To make this
formal we introduce the events 
\begin{align}
A_n:=&\big\{e^{-\lambda\delta n}Y_\0^p(\delta n)\in(pe^{-\lambda\delta},pe^{\lambda\delta})\big\},\\
B_n:=&\big\{e^{-\lambda t}Y_\0^p(t)\ge p(1-2\delta)e^{-2\lambda\delta}\text{ for all }t\in[\delta n,\delta(n+1)]\big\},\\
C_n:=&\big\{e^{-\lambda t}Y_\0^p(t)\le p(1+2\delta)e^{2\lambda\delta}\text{ for all }t\in[\delta n,\delta(n+1)]\big\}.
\end{align}
It will thus suffice to show that for every $\delta\in(0,\frac12)$ 
\as{} the events
$B_n$ and $C_n$ will occur for all but finitely many $n$. 
Let $M:=Y_\0^p(\gd n)$ be the number of balls present at the origin at time
$\delta n$, 
and let $y_n:=pe^{\lambda\delta(n-1)}$.
On the event
$A_n\cap B_n^c$,  $M >y_n$, and of these $M$ balls at
least $M-(1-2\delta) y_n$ must
die before time $\delta(n+1)$. Since each balls dies with probability
$1-e^{-\gd}<\delta$, 
Chebyshev's inequality implies, conditioned on $M>y_n$,
\begin{align}\label{pAB}
  \Pr(A_n\cap B_n^c\mid M) 
\le \frac{M}{\bigpar{M-(1-2\gd)y_n-\gd M}^2}
= \frac{M}{\bigpar{(1-\gd)M-(1-2\gd)y_n}^2}.
\end{align}
The right-hand side is decreasing in $M\ge y_n$, and thus, for all such $M$,
\CCreset
\begin{align}\label{pAB2}
  \Pr(A_n\cap B_n^c\mid M) \le \frac{y_n}{\bigpar{(1-\gd)y_n-(1-2\gd)y_n}^2}
=C y_n\qw = C' e^{-\gl\gd n}.
\end{align}
Furthermore, this holds trivially for $M<y_n$ too, since the conditional
probability then is 0.
Consequently, $  \Pr(A_n\cap B_n^c) \le C' e^{-\gl\gd n}$, and
the Borel--Cantelli lemma shows that \as{} the event $A_n\cap B_n^c$ occurs for
only finitely many $n$.

Similarly, 
with $y'_n:=pe^{\lambda\delta(n+2)}$,
on the event $A_{n+1}\cap C_n^c$ the number of balls
at the origin exceeds $(1+2\delta)y_n'$ at some point
during the time interval.
Let $\tau$ be the first time that this happens,
and $N\ge(1+2\delta)y'_n$ the number of balls at that time.
At least $N- y'_n$ of these balls must die before time $\gd(n+1)$.
Conditioned on $\tau$ and $N$, each ball dies with probability less than
$\delta$, and Chebyshev's inequality yields, for $N\ge y''_n:=(1+2\gd)y'_n$,
\begin{align}\label{pAC}
  \Pr(A_{n+1}\cap C_n^c\mid\tau,N) 
\le \frac{N}{\bigpar{(1-\gd)N-y'_n}^2}
\le \frac{y''_n}{\bigpar{(1-\gd)y''_n-y_n}^2}
=C'' y_n\qw = C''' e^{-\gl\gd n}.
\end{align}
Hence 
$  \Pr(A_{n+1}\cap C_n^c) \le C'''e^{-\gl\gd n}$, so
 the event $A_{n+1}\cap C_n^c$ occurs for only finitely many $n$.

Since $A_n$ \as{}
occurs for all large $n$ by~\eqref{delta-as}, it follows that for every
$\delta>0$ \as{}
\begin{equation}
p(1-2\delta)e^{-2\lambda\delta}\le\liminf_{t\to\infty}e^{-\lambda t}Y_\0^p(t)\le\limsup_{t\to\infty}e^{-\lambda t}Y_\0^p(t)\le p(1+2\delta)e^{2\lambda\delta}.
\end{equation}
This completes the proof of Theorem~\ref{thm:Zd-limit} in this more general setting.

Next, we see how to adapt the proof of Lemma~\ref{Llim}.
We need a bound on the event $\mathcal{E}_r(T)$.
For times $t$ such that $Z_\0(t,\bgz^{-,r})<Z_\0(t,\bgz^{+,r})$ and $Z_\0(t,\bgz^{+,r})>0$,
define the \emph{excess ball} (at time $t$) as the red ball 
with smallest label (in a fixed order) that is at $\0$ in
$\cZ(t,\bgz^{+,r})$ but does not exist in $\cZ(t,\bgz^{-,r})$;
if 
$Z_\0(t,\bgz^{-,r})<Z_\0(t,\bgz^{+,r})\le0$
define the {excess ball} as the blue ball 
with smallest label that is at $\0$ in
$\cZ(t,\bgz^{-,r})$ but does not exist in $\cZ(t,\bgz^{+,r})$.
For completeness, if $Z_\0(t,\bgz^{-,r})=Z_\0(t,\bgz^{+,r})$, define the excess
ball as an extra (non-existing) ball with its own clock.
Let $\tau$ denote the first time at which 
$Z_\0(t,\bgz^{-,r})<Z_\0(t,\bgz^{+,r})$ (with $\tau=\infty$ if this
never happens).
\refL{Lmon} gives~\eqref{eq:Lbounda} as before, and since
no two nucleations occur simultaneously, note that
\as{} $\mathcal{E}_r(T)=\{\tau\le T\}$. 

Let $F$ be the event that the excess ball 
does not die in the interval $(\tau,T]$. 
Then, on the event $\mathcal{E}_r(T)\cap F$,
\eqref{eq:Lbound} holds for $t=\tau$ and an
induction argument as in the proof of Lemma~\ref{Lmon} implies that
also \eqref{we} holds.
%\begin{equation}
%\cZ(T,\bgz^{\le r_1})+\delta_{\x,\0}\le \cZ(T,\bgz^{\le r'}).
%\end{equation}
Consequently, Markov's inequality yields that the right-hand side of
\eqref{tho} is a bound  for
$\Pr(\mathcal{E}_r(T)\cap F)$. 
To complete the proof it suffices to note that, since $\tau$ is a stopping
time and $\cE_r(T)$ is determined by $\tau$,
\begin{align}
 \Pr(F\mid \tau) = e^{-(T-\tau)_+}\ge e^{-T} 
\end{align}
and 
\begin{equation}
\Pr\bigpar{\mathcal{E}_r(T)\cap F}
\,=\,\E\bigl[\ind_{\mathcal{E}_r(T)}\Pr(F\mid\tau)\bigr]
\,\ge\, e^{-T}\Pr\bigl(\mathcal{E}_r(T)\bigr).
\end{equation}
The rest of the proof is the same as before.

Finally, we see how to prove Lemma~\ref{L0} in the more general setting. As before, it will suffice to consider the monochromatic process $\cY^p(t)$. Let $N$ denote the number of balls (in the monochromatic process) that arrive at the origin during the interval $[a,b]$, and $D$ the number of balls already present at time $a$ that die before time $b$. Then, using Markov's inequality as before,
\begin{equation}
 \Pr\bigl(Y_\0^p(t) \text{\rm\ is {not} constant for $t\in [a,b]$}\bigr)\le\E[N]+\E[D].
\end{equation}
Let $D'$ the number of balls that arrive after time $a$ and die before time $b$, and note that
\begin{equation}
\E[D]\le(b-a)\E[Y^p_\0(a)]\quad\text{and}\quad\E[D']\le(b-a)\E[N].
\end{equation}
In addition, $N= Y_\0^p(b)-Y_\0^p(a)+D+D'$, so under the assumption that
$b-a\le1/2$,
\begin{align}
\E[N]\le2\E[Y_\0^p(b)-Y_\0^p(a)+D]\le 2(e^{\lambda b}-e^{\lambda a})+2(b-a)e^{\lambda a}\le4(b-a)(\lambda+1)e^{\lambda T}.
\end{align}
The rest is silence.

\section{Open problems and further directions}\label{sec:open}

We round off with some open problems and suggested directions for further study, inspired by the results above. We give also some comments on possible extensions, some of which seem easy, but we leave them for the reader to check.

The problems may be considered for general branching rules, much like in the present paper, but in some cases (such as for the first question) it may make more sense for a specific branching rule (such as the nearest-neighbour rule, in which $\varphi$ is the deterministic configuration that puts a ball at each of the $2d$ neighbours of the origin). In some cases we even expect that the answer to the question may depend on the branching rule, much opposed to the results reported in this paper.

\begin{enumerate}

\item For $d=1$, what is the length of a typical monochromatic interval?

\item For $p>1/2$, at what rate does the density of blue sites tend to zero?

\item For $p=1/2$, at what rate does a site change colour?

\item For $p=1/2$, how may balls are contained at the origin at a given
  time? Proposition~\ref{prop:deviation} provides a partial answer, and Lemma~\ref{LVarZ} a matching upper bound. Is it true that $|Z_\org(t)|=\Theta(t^{-d/4}e^{\lambda t})$ with high probability, or does the density of times for which it holds tend to 1 as $t\to\infty$?
  Our arguments do not even seem to give the weaker conclusion that
$\Pr\bigr(Z_\0(t)=0\bigr)\to0$, and thus (by symmetry) 
  $ \rho(t)=\Pr\bigl(Z_\0(t)>0\bigr)\to1/2$ as $t\to\infty$. 
  
\item 
Are the moment conditions in Theorems \ref{thm:fixnonfix} and
\ref{thm:Zd-limit} necessary? In particular, is the condition $\E[\|\varphi\|^2]<\infty$, instead of $\E[\|\varphi\|^3]<\infty$,
sufficient for the conclusion of Theorem~\ref{thm:Zd-limit} to hold when $d=1$?

\item We have in this paper considered competition between two types. It would be interesting to extend our results to three or more competing types. We believe that it may be challenging to find a substitute for the conservative process described in Section~\ref{sec:coupling}. Problems of a similar character were suggested also in~\cite{ahlgrijanmor19}.

\item 
We assumed throughout the paper that the initial configuration has at most
one ball at each site. We can more generally consider initial configurations
$(\zeta_\x)_{\x\in\Zd}$ where the $\zeta_\x$ are \iid{} with an arbitrary
distribution.
We expect that the results above generalize rather easily under some moment
condition on $\gz_\x$, but we have not checked the details. We expect that it is less straightforward to adapt our techniques to allow the different types to jump at different rates, or reproduce according to different rules.

\item In the model studied by Bramson and
  Lebowitz~\cite{braleb91a,braleb91b}, no particles are born, and particles move
  according to independent continuous-time symmetric random walks. 
This can be regarded as an extreme case of our model (not covered above),
where balls die as in \refS{s:death} and the offspring $\gf$
consists of a single ball. 
For this
  model Cabezas, Rolla and Sidoravicius~\cite{cabrolsid18} have shown that,
  under weak assumptions, \as{} there exist arbitrarily large times when the
  origin is occupied.
A more detailed conjecture, which seems to be open,
  would be that, starting from a $p$-random Bernoulli initial colouring, the
  origin is \as{} visited by both colours infinitely many times when
  $p=1/2$, but not when $p>1/2$.

\item Consider the urn process on $\Z^d$ run from an initial configuration
  with a single red and a single blue ball. Under what conditions will
  both red and blue balls remain in the system at all times (so-called coexistence)
  with positive probability? For $d=1$, in the event of
  coexistence, under what conditions does an `interface', that is a
  macroscopic division, between red and blue exist, and how does it evolve
  over time? What is the analogous higher-dimensional phenomenon? Some
  progress have been made to these questions for a related model by Ahlberg,
  Angel and Kolesnik~\cite{ahlangkol}.

\end{enumerate}

\section*{Acknowledgements}
The authors are very grateful to Robert Morris, for
    his encouragement to pursue this project, and his valuable input in
    several joint discussions. The authors are also grateful to
    Luca Avena and Conrado da Costa for informing
    about the work of Pruitt~\cite{pruitt66} (cf. Remark~\ref{prutt}), to
    IMPA
 and to the Isaac Newton Institute, % for Mathematical Sciences
where parts of this work were done,
and to an anonymous referee who found a gap in our argument.

This work was in part supported by 
grant 2016-04442 from the Swedish Research Council (DA);
CNPq bolsa de produtividade Proc.~310656/2016-8 
and FAPERJ Jovem cientista do nosso estado Proc.~202.713/2018 (SG); 
the Knut and Alice Wallenberg Foundation,
the Isaac Newton Institute for Mathematical Sciences (EPSRC Grant Number
EP/K032208/1), 
and the Simons foundation (SJ).

%\bibliographystyle{abbrv}
%\bibliography{ahlberg}
%\newcommand{\noopsort}[1]{}\def\cprime{$'$}

\medskip

{\small
\noindent
{\sc Daniel Ahlberg\\
Department of Mathematics, Stockholm University\\ 
SE-10691 Stockholm, Sweden}\\
\url{http://staff.math.su.se/daniel.ahlberg/}\\
\texttt{daniel.ahlberg@math.su.se}\\

\noindent
{\sc Simon Griffiths\\
Departamento de Matem\'atica, PUC-Rio \\
Rua Marqu\^{e}s de S\~{a}o Vicente 225, G\'avea, 22451-900 Rio de Janeiro, Brasil}\\
\texttt{simon@mat.puc-rio.br}\\

\noindent
{\sc Svante Janson\\
Department of Mathematics, Uppsala University\\
PO Box 480\\
SE-75106 Uppsala, Sweden}\\
\url{http://www2.math.uu.se/~svante}\\
\texttt{svante.janson@math.uu.se}\\
}

\end{document}